\documentclass[11pt]{article}

\def \BasePath{./}

\RequirePackage{fix-cm}
\usepackage{graphicx}
\usepackage{amsmath}
\usepackage{amsthm}
\usepackage{amsfonts}
\usepackage{amssymb}
\usepackage{graphicx}
\usepackage{color}
\usepackage{rotating}
\usepackage{authblk}
\usepackage{url}
\usepackage{tcolorbox}
\usepackage{tikz,pgfplots} 
\usetikzlibrary{arrows,shadings,shadows,automata,shapes,calc,positioning,patterns,decorations.markings,decorations.pathmorphing,plotmarks}

\usepackage{soul}
\usepackage{footnote}
\usepackage{geometry}
\usepackage{subfigure}

\usepackage{lscape,tabularx}
\usepackage{printlen}
\newcolumntype{P}[1]{>{\centering\arraybackslash}p{#1}}
\newcommand\Tstrut{\rule{0pt}{2.6ex}}         
\newcommand\Bstrut{\rule[-0.9ex]{0pt}{0pt}}   

\usepackage[labelfont=bf]{caption}

\newtheorem{theorem}{Theorem}
\newtheorem{proposition}{Proposition}
\newtheorem{definition}{Definition}
\newtheorem{corollary}{Corollary}
\newtheorem{example}{Example}

\usepackage{natbib}
\bibpunct[, ]{(}{)}{,}{a}{}{,}%

\usepackage{authblk}
\usepackage{lipsum}

\makeatletter
\renewcommand\footnoterule{%
  \kern-3\p@
  \hrule\@width \textwidth
  \kern2.6\p@}
\makeatother

\makeatletter
\renewcommand*{\@fnsymbol}[1]{\ensuremath{\ifcase#1\or *\or \dagger\or \ddagger\or **\or \mathsection\or \mathparagraph\or \|\or  \dagger\dagger
   \or \ddagger\ddagger \else\@ctrerr\fi}}
\makeatother

\usepackage{hyperref}	 
\hypersetup{
 colorlinks={true},linkcolor={magenta},citecolor=blue!60!black,urlcolor  = blue!60!black,
 }

\newcommand{\parentheses}[1]{\left(#1\right)}

\newcommand{\singleindexset}[1]{\left[ #1 \right]}
\newcommand{\set}[1]{\left\{#1\right\}}
\newcommand{\bigO}{\mathcal{O}}

\newcommand{\modo}{\mathcal{M}}
\newcommand{\modorec}{\modo_{\mathcal{R}}}
\newcommand{\feasibleset}{\mathcal{X}}
\newcommand{\feasiblevaluestate}{\mathcal{V}}
\newcommand{\nobj}{K}
\newcommand{\objindex}{k}
\newcommand{\xset}{\mathcal{X}}
\newcommand{\yset}{\mathcal{Y}}

\newcommand{\nds}{\yset_\textnormal{N}}
\newcommand{\nd}[1]{\mathtt{ND}\parentheses{#1}}
\newcommand{\strictlydominated}{\prec}

\newcommand{\strictlydominates}{\succ}

\newcommand{\paretofrontier}[1]{\mathtt{PF}\parentheses{#1}}
\newcommand{\statespace}{\mathcal{S}}

\newcommand{\reward}{\delta}
\newcommand{\transition}{\tau}

\newcommand{\networkmodel}{\mathcal{N}}
\newcommand{\nodes}{\mathcal{L}}
\newcommand{\arcs}{\mathcal{A}}
\newcommand{\layer}[1]{\mathcal{L}_{#1}}
\newcommand{\terminalnode}{\mathbf{t}}
\newcommand{\rootnode}{\mathbf{r}}
\newcommand{\arcdomain}[1]{d\parentheses{#1}}
\newcommand{\arcweights}[1]{w\parentheses{#1}}

\newcommand{\arcweightfunction}{w}
\newcommand{\arcroot}[1]{r\parentheses{#1}}
\newcommand{\arcterminal}[1]{t\parentheses{#1}}
\newcommand{\arclayer}[1]{\ell\parentheses{#1}}
\newcommand{\nodelayer}[1]{\ell\parentheses{#1}}

\newcommand{\pathset}[2]{\mathcal{P}\parentheses{#1,#2}}
\newcommand{\Pathset}{\mathcal{P}}

\newcommand{\pathweight}[1]{w\parentheses{#1}}
\newcommand{\labelSetTD}[1]{\mathcal{Z}^\mathrm{TD}\parentheses{#1}}
\newcommand{\labelSetBU}[1]{\mathcal{Z}^\mathrm{BU}\parentheses{#1}}
\newcommand{\couple}[2]{\mathtt{CP}\parentheses{#1,#2}}
\newcommand{\coupleLayer}[1]{\mathtt{CP^L}\parentheses{#1}}

\DeclareMathOperator*{\argmax}{arg\,max}

\newcommand{\kirlik}{\texttt{K}}
\newcommand{\aira}{\texttt{O}}
\newcommand{\bu}{\texttt{BU}}
\newcommand{\bup}{\texttt{BU+}}
\newcommand{\td}{\texttt{TD}}
\newcommand{\tdp}{\texttt{TD+}}
\newcommand{\coup}{\texttt{Coup}}
\newcommand{\coupp}{\texttt{Coup+}}

\newcommand{\andre}[1]{\textcolor{black}{#1}} 
\newcommand{\carlos}[1]{\textcolor{black}{#1}} 

\begin{document}

\title{Network Models for Multiobjective Discrete Optimization}

\author[1]{David Bergman\thanks{david.bergman@business.uconn.edu}}
\author[2]{Merve Bodur\thanks{bodur@mie.utoronto.ca}}
\author[3]{Carlos Cardonha\thanks{carloscardonha@br.ibm.com}}
\author[4]{Andre A. Cire\thanks{acire@utsc.utoronto.ca}} 
\affil[1]{\small Department of Operations and Information Management, University of Connecticut} 
\affil[2]{\small Department of Mechanical and Industrial Engineering, University of Toronto}
\affil[3]{\small IBM Research}
\affil[4]{\small Department of Management, University of Toronto Scarborough}
\maketitle

\begin{abstract}
This paper provides a novel framework for solving multiobjective discrete
optimization problems with an arbitrary number of objectives. Our framework
formulates these problems as \textit{network models}, in that enumerating the
Pareto frontier amounts to solving a multicriteria shortest path problem in an
auxiliary network. We design techniques for exploiting the network
model in order to accelerate the identification of the Pareto frontier, most
notably a number of operations to simplify the network by removing nodes and
arcs while preserving the set of nondominated solutions. We show that the
proposed framework yields orders-of-magnitude performance improvements over
existing state-of-the-art algorithms on five problem classes containing both
linear and nonlinear objective functions.
\end{abstract}

\section{Introduction}
\label{sec:introduction}


\textit{Multiobjective optimization} is the study of algorithms for optimization
problems associated with two or more objective functions. Such problems arise
naturally in prescriptive decision making, where decision makers often face
conflicting criteria that balance trade-offs associated with a proposed
solution. Practical applications of multiobjective optimization are pervasive
and found in a diverse set of domains. This includes, for example, production
planning for supply chains \citep{dickersbach2015supplychain}, radiotherapy
optimization in healthcare \citep{radiotherapyYu2009}, and aerodynamic design
\citep{aerodynamicWang2011}, which may involve from a few up to potentially
hundreds of objective functions. Furthermore, multiobjective optimization can
also be used as an alternative method to solve single-objective problems \citep{bodur2016decomposition}, expanding even further on
the importance of the study of techniques to tackle these
computationally challenging problems.

\carlos{
In this paper, we focus on 
multiobjective \textit{discrete} optimization problems (MODOs) that admit recursive
formulations; in such problems, variables may only assume values
from a finite set.} There is a rich history of the study of techniques for
addressing MODOs, often leveraging advances in integer programming (IP) technology; see,
e.g., in-depth surveys by \citealt{ehrgott2006discussion} and
\citealt{ZhoQuHiZhaSugZha11}. Specifically, current state-of-the-art
methodologies rely on parametric single-objective reformulations of MODOs,
employing commercial IP solvers as black-boxes of 
their algorithms. The Pareto frontier, however, cannot be fully
recovered by such linear parametric models in general \citep{sayin2005multiobjective}.
These methods therefore rely on explicit enumeration techniques, limiting the size of problems that can be solved both in terms of
the number of variables and the number of objective functions. 

\textit{Contributions.} This paper presents a new approach for modeling and
solving MODOs \carlos{that admit recursive formulations}. 
Our framework reformulates the problem of identifying points in
the Pareto frontier as a multicriteria shortest path problem
\citep{GarGioTav10} over a structured network. In particular, the network
implicitly represents the objective space in a compact way by exploiting
symmetry and dominance relationship between solutions, which can be defined both generally 
or in a problem-specific form.
 
Our first contribution is the design of two approaches for obtaining a valid
network for a MODO.  One is based on a recursive model and the other is
extracted from a decision diagram representation of the problem
\citep{bergman2016decision}. Both techniques leverage construction procedures
already available in the literature, some of which are inspired by early dynamic
programming (DP) approaches for multiobjective optimization
\citep{carraway1990generalized}. Nonetheless, we exploit the network structure
as opposed to the recursive formulation directly, which provides two benefits. First,
network representations do not require linear formulations, thereby broadening
modeling expressiveness over other popular techniques. Second, network models
inherently exploit symmetry by combining subpaths that correspond to common
objective function evaluations, decreasing the enumeration requirements from
previous techniques.


As our second contribution, we propose \textit{validity-preserving operations}
(VPOs) designed to reduce the size of a network while maintaining validity. A
network for a MODO is, in general, exponentially large in the input size of the
problem.  Even if the network itself is of manageable size, computing a
multicriteria shortest path may take a prohibitively long time.  VPOs simplify
the network without modifying the Pareto frontier, leading to significant
reductions on the number of arcs and nodes and hence computational time. We
explore VPOs that are based solely on the network itself (e.g., removing
arcs/nodes, merging nodes) and VPOs that explore 
domain-specific features of the problem (e.g., using DP state-based
information to identify dominance).  

Finally, we present an extensive numerical study on five problem classes to
compare our network model approach with two state-of-the-art MODO algorithms. We
consider the knapsack, set covering, set packing, and the traveling salesperson
problem, which are commonly used as benchmarks in the MODO literature, in
addition to a MODO with nonlinear terms in the objective motivated from an 
application in regression-based models. 

\andre{ Our experiments indicate that the proposed approach outperforms the
state-of-the-art by orders of magnitude \carlos{for problem classes where recursive models and state-space relaxations are known to be effective; in particular, the results show a significant expansion} on the size of problems that
can be solved,  specifically in terms of the number of objective functions (up
to seven). This is a particularly limiting factor in existing approaches that
severely restrict the applicability of multiobjective optimization in real scenarios beyond a
few objectives, as highlighted by \cite{Duro2017}. Examples of applications
where the Pareto frontier for four or more objectives is desired include protein
structure prediction \citep{Regina2013}, computational sustainability
\citep{Gomes2018}, storm drainage and work roll cooling \citep{Duro2017}, to
name a few. In practice, the Pareto frontier fully characterizes the trade-offs
among solutions and is post-processed in interactive decision support systems
\citep{stewart2008real}. This allows practitioners to prioritize objectives and
operational aspects based on their technical expertise, often on a case-by-case
basis, which can be used as opposed to or in conjunction with typical scalarization techniques
when multiple objectives are present.
}

The remainder of this paper is organized as follows.
\S\ref{sec:literatureReview} provides a literature review of MODO, specifically
as it relates to the present paper.
\S\ref{sec:multi-objectiveDiscreteOptimization} and \S\ref{sec:network-models}
formally define MODOs and network models, respectively.
\S\ref{sec:NMConstructionAlgos} describes network construction algorithms.
\S\ref{sec:VPOs} presents VPOs, and  \S\ref{sec:findingTheParetoFrontier}
presents the multicriteria shortest path algorithms we employ for enumerating the Pareto frontier.
\S\ref{sec:copmutationalInsights} describes the results of an experimental
evaluation on four problem classes.
We conclude and describe future work in \S\ref{sec:conclusionAndFutureWork}. Proofs which do not appear in the main text are presented in the appendix, ~\S\ref{sec:proofs}.

There is an extensive literature on exact algorithms for generating the Pareto
frontier of a MODO. In general, these approaches can be divided into two main
classes: those based on \emph{criterion-space search}, and those based on
\emph{decision-space search} \citep{ehrgott2016exact,ehrgott2006discussion}.

Criterion-space search relies on \emph{scalarizations}, most commonly based on
a combination of weighted sums and $\epsilon$-constraints. Weighted-sum methods
iteratively solve a single-objective optimization version of the problem where
the single objective is defined by various positive-weight combinations of the
original objectives. For general MODOs, however, only a portion of the Pareto
frontier (those points referred to as \emph{supported} efficient points) can be
identified by this approach alone. The remaining points (\emph{unsupported}
efficient points) can be found through the $\epsilon$-constraint method,
introduced by \cite{haimes1971bicriterion}, which optimizes one of the original
objective functions with the other constraints transformed into parametrized
constraints. 

\cite{kirlik2014new} and \cite{Ozlen2013} provide the state-of-the-art
criterion-space search algorithms for MODOs with an arbitrary number of
objectives. Both algorithms are based on variants of scalarization techniques.
These variants build upon the early work by \cite{klein1982algorithm}, who
suggested an iterative approach to generate a subset of the Pareto frontier
while refining the search space by the addition of disjunctive constraints.
\cite{sylva2004method} extended this work to an exact algorithm by
reformulating the disjunctive conditions as big-$M$ constraints, which was
further improved by \cite{lokman2013finding} and \cite{bektas2016disjunctive}.
extended to more than two objectives by \cite{tenfelde2003recursive} and
further enhanced by \cite{dhaenens2010k}. Another generalization of the
two-phase method is proposed by \cite{przybylski2010two}.
\cite{ozlen2009multi}, in turn, developed an alternative approach called the
\textit{augmented} $\epsilon$-constraint method, which became one of the
state-of-the-art methods after later enhancements by \cite{Ozlen2013}. Another
improvement to the augmented $\epsilon$-constraint method is the recursive
methodology proposed by \cite{laumanns2005adaptive,laumanns2006efficient},
further refined by \cite{kirlik2014new} into a computationally practical
approach. We note that \cite{boland2016new} developed an extension of the
so-called L-shape search method (specific to triobjective MODOs) that optimizes
a linear function over the set of nondominated points, that can also be used to
enumerate the Pareto frontier.

Other scalarization methods include the (lexicographic or augmented) weighted Tchebycheff scalarization \citep{steuer1983interactive}. The majority of the criterion-space search focuses on \emph{biobjective} problems, where the special structure resulting from only having two objectives can be exploited; see, e.g.,  \cite{ralphs2006improved,sayin2005multiobjective,boland2015criterion,parragh2015branch}. Extensions of these ideas have also been proposed for triobjective MODOs, which iteratively decompose the search space into smaller regions, and apply efficient ways to explore and refine these regions \citep{dachert2015linear,BolandLSM2016,boland2016quadrant}.

Decision-space search methods operate over the space defined by the original decision variables. These techniques are typically based on branch-and-bound search developed for mixed-integer linear programs. The first of such algorithms was proposed by \cite{mavrotas1998branch} for binary MODOs. The algorithm uses an artificial ideal point to define a bounding set, and discovers nondominated points by solving (via a criterion-space search algorithm) the multiobjective linear programs obtained when all binary variables are fixed. \cite{mavrotas2005multi} observed that this branch-and-bound in fact generates a superset of the Pareto frontier, and proposed filtering algorithms to eliminate spurious points. \cite{Vincent2013} later showed that the previous algorithm is still incomplete, and proposed a corrected and improved version for biobjective problems. Other branch-and-bound techniques have also been studied by \cite{masin2008diversity,sourd2008bb}, who suggested enhancements to the bounding aspects. 


The first decision-space algorithm for generic MODO based on branch-and-cut was developed by \cite{jozefowiez2012}, where discrete sets are used for lower bounds.   Recently, \cite{adelgren2017branch} developed a new branch-and-bound algorithm which employs multiobjective extensions of many different aspects of branch-and-bound, such as (primal and dual) presolve, preprocessing, node processing, and dual bounding via cutting planes.

Alternative methods that avoid scalarizations are based on DP and implicit
enumerative methods for pure binary problems (e.g.,
\citealt{bitran1977linear,bitran1982combined}). DP approaches are typically
focused on variants of the multiobjective knapsack problem, such as
\cite{villarreal1981multicriteria} who employed lower and upper bound sets to
eliminate dominated solutions. The authors extended their work to general
stage-wise separable MODOs \citep{villarreal1982multicriteria}.
\cite{klamroth2000dynamic} presented distinct conceptual DP models for several
variants of the multiobjective knapsack problem.
\cite{bazgan2009solving} developed a new DP approach enhanced with complementary
dominance relations for the 0-1 knapsack case. For the biobjective knapsack,
\cite{delort2010using} and \cite{rong2014dynamic} proposed a two-phase algorithm
and a multiobjective DP algorithm, respectively.

Similar to the above DP approaches, the most relevant works to ours also
considered the multiobjective knapsack problem. \cite{captivo2003solving}
proposed a transformation of biobjective 0--1 knapsack problem into a
biobjective shortest path problem, which is solved via an enhanced version of
the labeling algorithm for MSPs. \cite{figueira2010labeling} developed a generic
labeling algorithm for the problem that applies existing reformulations from the literature.
More recently, state reduction techniques for the biobjective case have been
proposed by \cite{rong2011two} and \cite{rong2013reduction}, with algorithmic
enhancements by \cite{figueira2013algorithmic}.

\section{Multiobjective Discrete Optimization Problems}
\label{sec:multi-objectiveDiscreteOptimization}

\andre{
	In this section we present the notation and formalism used throughout
the text. For $a \in \mathbb{N}_+$, we let $[a] := \set{1, 2, \ldots, a}$.
We denote by $\boldsymbol{0}$ and $\boldsymbol{1}$, a vector of
zeroes and ones in appropriate dimention, respectively. The notation
$\mathbb{B} := \{0,1\}$ indicates the Boolean set, while the operator
$(\cdot)^\top$ denotes the transpose. For a given $v \in \mathbb{R}^p$, we
denote the $i$-th component of $v$ by $v_i$ or $(v)_i$, for all $i \in [p]$. 
}

A multiobjective discrete optimization problem (MODO) is of the form
\begin{align*}
\tag{$\modo$}
\max \left \{ f(x) := \left(f_1(x), f_2(x), \dots, f_\nobj(x) \right) \, : \, x \in \xset \right \},
\end{align*}
\andre{where $\xset \subset \mathbb{Z}^n$, $n \in \mathbb{N}_+$, is a bounded} feasible set and $f : \xset \rightarrow \mathbb{R}^\nobj$
maps each solution~$x \in \xset$ into a $\nobj$-dimensional
\emph{objective vector} (or image) $y :=
\big(f_1(x),\hdots,f_\nobj(x)\big)$, with $f_\objindex : \xset
\rightarrow \mathbb{R}$, $k \in \singleindexset{\nobj}$. 

The objective functions are not assumed to have any particular structure, except
that they are well-defined in $\xset$. For any two objective vectors $y, y' \in
\mathbb{R}^\nobj$, we say that $y$ \textit{dominates} $y'$ 
(\carlos{or, alternatively, that $y'$ \textit{is dominated by} $y$}), or simply $y
x\strictlydominates y'$, if (i) $y_\objindex \geq y'_\objindex$ for all
$\objindex \in \singleindexset{\nobj}$, and (ii) there exists at least one index
$\tilde{\objindex}$ for which $y_{\tilde{\objindex}} > y'_{\tilde{\objindex}}$.



The image defined by the set of feasible solutions is
denoted by $\yset := \{ f(x) \,:\, x \in \xset \}$.
An objective vector $y^* \in \yset$ is a \emph{nondominated} point if there exists no other point $y' \in \yset$  for which $y' \strictlydominates y^*$. The set of all nondominated points of $\modo$ is denoted by $\nds$, also referred to as the \textit{Pareto frontier}. The typical goal of a MODO, and the focus of this paper, is to enumerate $\nds$.  
 
\begin{example}\label{ex:modo}
	We consider a set packing instance as a running example, with $\nobj = 3$:
	\begin{align*}
	\max_{x \in \mathbb{B}^7} \quad& \big(f_1(x) = 4x_1 + 5x_2 + 3x_3 + 4x_4 + 2x_5 + 1x_6 + 2x_7,\\
		 \quad& \,f_2(x) = 8x_1 + 7x_2 + 1x_3 + 5x_4 + 3x_5 + 3x_6 + 8x_7,\\
		 \quad& \,f_3(x) = 2x_1 + 6x_2 + 8x_3 + 4x_4 + 6x_5 + 5x_6 + 2x_7 \big)\\
	\textnormal{s.t.} \quad&
								x_1 + x_2 + x_3 \leq 1, \;\; x_2 + x_3  + x_4 \leq 1, \;\; x_4 + x_5 \leq 1, \\
					   \quad&   x_4 + x_6 \leq 1, \;\; x_5 + x_7 \leq 1, \;\; x_6 + x_7 \leq 1.
	\end{align*}
%
	The nondominated set~$\nds$ consists of the four points $y^1 = (6,7,19)$, 
	$y^2 = (8,13,17)$, $y^3 = (7,14,13)$, and $y^4 = (10,21,8)$. They 
are the images, respectively, of the feasible solution vectors
$x^1 = (0,0,1,0,1,1,0)$, $x^2 = (0,1,0,0,1,1,0)$, $x^3 = (1,0,0,0,1,1,0)$, and $x^4 = (1,0,0,1,0,0,1)$.
\hfill $\square$ 
\end{example}


For any $\bar{\yset} \subseteq \yset$, let 
$\nd{\bar{\yset}} := \set{y \in \bar{\yset}: \nexists\,y' \in \bar{\yset} \ \text{with} \ y' \strictlydominates y }$ 
be an operator that returns the set of vectors within $\bar{\yset}$ that are not dominated by any other vector in the same set. Note that $\nds = \nd{\yset}$. This operator	
has been studied in the context of relational database systems, where it is known as the \emph{skyline} operator \citep{borzsony2001skyline}. We refer to the work by \cite{gudala2012} for a review of algorithms to compute $\nd{\cdot}$ and their associated complexity analysis. In particular, for $\nobj = 2$, an efficient implementation of $\nd{\mathcal{S}}$ for a given $\mathcal{S} \subseteq \mathbb{R}^\nobj$ has a worst-case time complexity of $\mathcal{O}(|\mathcal{S}| \log\parentheses{|\mathcal{S}|})$. For $\nobj > 2$, it can be efficiently implemented in $\mathcal{O}\parentheses{|\mathcal{S}|\cdot \parentheses{\log\parentheses{|\mathcal{S}|}}^{\nobj - 2}}$.

%
%
%

\section{Network Models}
\label{sec:network-models}

This paper proposes the use of \textit{network models} to enumerate $\nds$ for
a MODO $\modo$. In our context, a network model is a layered-acyclic multi-digraph
$\networkmodel := \parentheses{\nodes,\arcs}$ with node set $\nodes$ and arc
set $\arcs$. Such a model is equipped with specific structure, properties, and
attributes associated with $\modo$.

The node set $\nodes$ is partitioned into $n+1$ non-empty layers $\nodes := \dot{\bigcup\limits_{j \in \singleindexset{n+1}}} \layer{j}$.  Layers $\layer{1}$ and $\layer{n+1}$ have cardinality one with $\layer{1} := \set{\rootnode}$ and $\layer{n+1} := \set{\terminalnode}$. Nodes $\rootnode$ and $\terminalnode$ are referred to as the \textit{root} node and the \textit{terminal} node, respectively.  The layer index of a node 
$u \in \nodes$ is $\nodelayer{u}$, i.e., $u \in \layer{\nodelayer{u}}$. Each arc $a := \parentheses{\arcroot{a},\arcterminal{a}} \in \arcs$ is directed from its \emph{arc-root} $\arcroot{a} \in \layer{j}$ to its \emph{arc-terminal} $\arcterminal{a} \in \layer{j+1}$ for some $j \in \singleindexset{n}$. The layer of an arc is $\arclayer{a} := \nodelayer{\arcroot{a}}$.  We denote by $\arcs^+(u) := \set{a \in \arcs: \arcroot{a} = u}$ the set of outgoing arcs from $u$ and by $\arcs^-(u) := \set{a \in \arcs: \arcterminal{a} = u}$ the set of incoming arcs to $u$. 

We define $\pathset{u}{v}$ as the set of arc-specified paths from node $u$ to node $v$. The arguments will be omitted when  $u = \rootnode$ and $v = \terminalnode$, i.e., $\Pathset := \pathset{\rootnode}{\terminalnode}$. Each arc $a$ has an \emph{arc-weight} vector $\arcweights{a} \in \mathbb{R}^\nobj$. The arc-weight vectors provide the connection between arc-specified paths in $\networkmodel$ and the objective space of $\modo$.  Any path $p = \parentheses{a_{j_1}, \ldots, a_{j_H}}$ has \emph{path-weight} $\pathweight{p} = \sum\limits_{h \in [1,H]} \arcweights{a_{j_h}}$.

The Pareto frontier of a network model $\networkmodel$ is defined
as
\[
\paretofrontier{\networkmodel} := 
\nd{ 
	\bigcup\limits_{p \in \Pathset} w(p)
}.
\]
A network model $\networkmodel$ is \emph{valid} for a MODO $\modo$ if $\paretofrontier{\networkmodel} = \nds$.  In our structural results, we may operate on distinct valid network models for the same $\modo$, in which case we append a subscript to our notation so as to indicate the network. For example, $\rootnode_{\networkmodel}$ will be used to represent the root node of $\networkmodel$. The subscript will be omitted when~$\networkmodel$ is clear from context.

A network model can be interpreted as a data structure that supports the
representation of a MODO $\modo$ as a multicriteria shortest path problem (MSP). By negating the arc-weight vectors
(since we consider maximization), one may apply any algorithm for solving MSPs
to a network model~$\networkmodel$  in order to find
$\paretofrontier{\networkmodel}$ 
if~$\networkmodel$ is valid for $\modo$. Several MSP algorithms are available;
see, e.g., the survey by \cite{GarGioTav10}.

\begin{example}\label{ex:bdd} Figure~\ref{fig:bddexample} depicts a network model $\networkmodel$ for the MODO in Example~\ref{ex:modo}. 
The arc-weight vectors are shown (in black text) next to each arc (the additional details provided in the figure, in red and blue, will be introduced and described in Example \ref{ex:unidirectional}). There are 14 arc-directed paths from $\rootnode$ to $\terminalnode$, and their path-weights are given by
	\begin{eqnarray*}
		\bigcup_{p \in \Pathset} \arcweights{p} 
			=
			\big\{
				\parentheses{8,13,17},
				\parentheses{6,10,11},
				\parentheses{7,15,8},
				\parentheses{6,7,19},
				\parentheses{4,4,13},
				\parentheses{5,9,10},
				\parentheses{3,6,11},
				\\
				\parentheses{1,3,5},
				\parentheses{2,8,2},
				\parentheses{7,14,13},
				\parentheses{6,13,6},
				\parentheses{5,11,7},
				\parentheses{6,16,4},
				\parentheses{10,21,8}
			\big\}
	\end{eqnarray*}
	and results in 
	\[
\paretofrontier{\networkmodel} = 
\nd{ 
	\bigcup\limits_{p \in \Pathset} w(p)
} = 
	\big\{
		\parentheses{8,13,17},
		\parentheses{6,7,19},
		\parentheses{7,14,13},
		\parentheses{10,21,8}
		\big\}.
\]
\hfill $\square$ 
\end{example}

\begin{figure}[h!]
	\centering
	\begin{tikzpicture}[scale=0.3][font=\sffamily,\tiny]
	\node [scale=1.2] (x1) at (-20,-2) {$x_1$};
	\node [scale=1.2] (x2) at (-20,-8) {$x_2$};
	\node [scale=1.2] (x3) at (-20,-14) {$x_3$};
	\node [scale=1.2] (x4) at (-20,-20) {$x_4$};
	\node [scale=1.2] (x5) at (-20,-26) {$x_5$};
	\node [scale=1.2] (x6) at (-20,-32) {$x_6$};
	\node [scale=1.2] (x7) at (-20,-38) {$x_7$};
	
	\node [draw,circle,scale=1] (r) at (0,0) {$\rootnode$};
	\node [color=red] () at (-3,0) { $\checkmark \parentheses{0,0,0}$ }; 
	\node [color=blue] () at (5.1,2) { $\checkmark  \parentheses{8,13,17}$ }; 
	\node [color=blue] () at (4.85,1) { $\times  \parentheses{7,15,8}$ };
	\node [color=blue] () at (4.9,0) { $\checkmark  \parentheses{6,7,19}$ };
	\node [color=blue] () at (5.1,-1) { $\checkmark  \parentheses{7,14,13}$ };
	\node [color=blue] () at (5.1,-2) { $\checkmark  \parentheses{10,21,8}$ };
	
	\node [draw,circle,scale=1] (a) at (-6,-5) {$u^2_1$};
	\node [color=blue] () at (-0.98,-3.1) { $\checkmark \parentheses{8,13,17}$ }; 
	\node [color=blue] () at (-1.2,-4) { $\checkmark \parentheses{7,15,8}$ }; 
	\node [color=blue] () at (-1.2,-5) { $\checkmark \parentheses{6,7,19}$ }; 
	\node [color=blue] () at (-1.2,-6) { $\times \parentheses{5,9,10}$ }; 
	\node [color=blue] () at (-1.2,-7) { $\times \parentheses{6,13,6}$ }; 
	\node [color=blue] () at (-1.2,-8) { }; 
	\node [color=red] () at (-9,-5) { $\checkmark \parentheses{0,0,0}$ }; 
	
	\node [draw,circle,scale=1] (b) at (6,-5) {$u^2_2$};
	\node [color=red] () at (3,-5) { $\checkmark \parentheses{4,8,2}$ }; 
	\node [color=blue] () at (9,-4.5) { $\checkmark \parentheses{3,6,11}$ }; 
	\node [color=blue] () at (9,-5.5) { $\checkmark \parentheses{6,13,6}$ }; 
	
	\node [draw,circle,scale=1] (c) at (-12,-11) {$u^3_1$};
	\node [color=red] () at (-15,-11) { $\checkmark \parentheses{5,6,7}$ };
	\node [color=blue] () at (-8.8,-10.5) { $\checkmark \parentheses{3,6,11}$ };
	\node [color=blue] () at (-9,-11.5) { $\checkmark \parentheses{2,8,2}$ };
	
	\node [draw,circle,scale=1] (d) at (0,-11) {$u^3_2$};
	\node [color=red] () at (-3,-11) { $\checkmark \parentheses{0,0,0}$ };
	\node [color=blue] () at (3.1,-8.75) { $\checkmark \parentheses{6,7,19}$ };
	\node [color=blue] () at (3.1,-9.75) { $\checkmark \parentheses{5,9,10}$ };
	\node [color=blue] () at (3.1,-10.75) { $\times \parentheses{3,6,11}$ };
	\node [color=blue] () at (3.1,-11.75) { $\checkmark \parentheses{6,13,6}$ };

	\node [draw,circle,scale=1] (e) at (12,-11) {$u^3_3$};
	\node [color=red] () at (9,-11) { $\checkmark \parentheses{4,8,2}$ }; 
	\node [color=blue] () at (15,-10.5) { $\checkmark \parentheses{3,6,11}$ }; 
	\node [color=blue] () at (15,-11.5) { $\checkmark \parentheses{6,13,6}$ }; 
	
	\node [draw,circle,scale=1] (f) at (-6,-17) {$u^4_1$};
	\node [color=red] () at (-9,-16.5) { $\checkmark \parentheses{5,7,6}$ };
	\node [color=red] () at (-9,-17.5) { $\checkmark \parentheses{3,1,8}$ };
	\node [color=blue] () at (-2.8,-16.5) { $\checkmark \parentheses{3,6,11}$ };
	\node [color=blue] () at (-3,-17.5) { $\checkmark \parentheses{2,8,2}$ };
	
	\node [draw,circle,scale=1] (g) at (6,-17) {$u^4_2$};
	\node [color=red] () at (3,-16.5) { $\checkmark \parentheses{4,8,2}$ };
	\node [color=red] () at (3,-17.5) { $\times \parentheses{0,0,0}$ };
	\node [color=blue] () at (9.2,-16) { $\checkmark \parentheses{3,6,11}$ };
	\node [color=blue] () at (9,-17) { $\times \parentheses{2,8,2}$ };
	\node [color=blue] () at (9.2,-18) { $\checkmark \parentheses{6,13,6}$ };
	
	\node [draw,circle,scale=1] (h) at (-6,-23) {$u^5_1$};
	\node [color=red] () at (-9,-22) { $\checkmark \parentheses{5,7,6}$ };
	\node [color=red] () at (-9,-23) { $\checkmark \parentheses{3,1,8}$ };
	\node [color=red] () at (-9,-24) { $\checkmark \parentheses{4,8,2}$ };
	\node [color=blue] () at (-1.5,-22) { $\checkmark \parentheses{3,6,11}$ };
	\node [color=blue] () at (-1.7,-23) { $\checkmark \parentheses{2,8,2}$ };
	\node [color=blue] () at (-1.7,-24) { $\times \parentheses{1,3,5}$ };

	\node [draw,circle,scale=1] (i) at (6,-23) {$u^5_2$};
	\node [color=red] () at (3,-23) { $\checkmark \parentheses{8,13,6}$ };
	\node [color=blue] () at (9,-23) { $\checkmark \parentheses{2,8,2}$ };
	
	\node [draw,circle,scale=1] (j) at (-12,-29) {$u^6_1$};
	\node [color=red] () at (-14.8,-28) { $\checkmark \parentheses{7,10,12}$ };
	\node [color=red] () at (-15,-29) { $\checkmark \parentheses{5,4,14}$ };
	\node [color=red] () at (-15,-30) { $\checkmark \parentheses{6,11,8}$ };
	\node [color=blue] () at (-9,-29) { $\checkmark \parentheses{1,3,5}$ };
	
	\node [draw,circle,scale=1] (k) at (0,-29) {$u^6_2$};
	\node [color=red] () at (-3,-28) { $\checkmark \parentheses{5,7,6}$ };
	\node [color=red] () at (-3,-29) { $\checkmark \parentheses{3,1,8}$ };
	\node [color=red] () at (-3,-30) { $\checkmark \parentheses{4,8,2}$ };
	\node [color=blue] () at (3,-28.5) { $\checkmark \parentheses{2,8,2}$ };
	\node [color=blue] () at (3,-29.5) { $\checkmark \parentheses{1,3,5}$ };
	
	\node [draw,circle,scale=1] (l) at (12,-29) {$u^6_3$};
	\node [color=red] () at (9,-29) { $\checkmark \parentheses{8,13,6}$ };
	\node [color=blue] () at (14.8,-29) { $\checkmark \parentheses{2,8,2}$ };
	
	\node [draw,circle,scale=1] (m) at (-6,-35) {$u^7_1$};
	\node [color=red] () at (-11.8,-32.5) { $\checkmark \parentheses{8,13,17}$ };
	\node [color=red] () at (-12,-33.5) { $\checkmark \parentheses{6,7,19}$ };
	\node [color=red] () at (-11.8,-34.5) { $\checkmark \parentheses{7,14,13}$ };
	\node [color=red] () at (-11.9,-35.5) { $\times \parentheses{6,10,11}$ };
	\node [color=red] () at (-12.1,-36.5) { $\times \parentheses{4,4,13}$ };
	\node [color=red] () at (-12.1,-37.5) { $\times \parentheses{5,11,7}$ };
	\node [color=blue] () at (-3,-35) { $\checkmark \parentheses{0,0,0}$ };
	
	\node [draw,circle,scale=1] (n) at (6,-35) {$u^7_2$};
	\node [color=red] () at (2.1,-33.5) { $\times \parentheses{5,7,6}$ };
	\node [color=red] () at (2,-34.5) { $\checkmark \parentheses{3,1,8}$ };
	\node [color=red] () at (2,-35.5) { $\times \parentheses{4,8,2}$ };
	\node [color=red] () at (2.14,-36.5) { $\checkmark \parentheses{8,13,6}$ };
	\node [color=blue] () at (8.8,-35) { $\checkmark \parentheses{2,8,2}$ };
	
	\node [draw,circle,scale=1] (t) at (0,-41) {$\terminalnode$};
	\node [color=red] () at (-4.8,-39) { $\checkmark \parentheses{8,13,17}$ };
	\node [color=red] () at (-5,-40) { $\checkmark \parentheses{6,7,19}$ };
	\node [color=red] () at (-4.9,-41) { $\checkmark \parentheses{7,14,13}$ };
	\node [color=red] () at (-5.15,-42) { $\times \parentheses{5,9,10}$ };
	\node [color=red] () at (-4.9,-43) { $\checkmark \parentheses{10,21,8}$ };
	\node [color=blue] () at (2.8,-41) { $\checkmark \parentheses{0,0,0}$ };
	
	
	\path[->](r) edge node [xshift=-3.4mm,yshift=-2.6mm,left] {$\parentheses{0,0,0}$} (a);
	\path[->](r) edge node [xshift=3.5mm,yshift=-2.6mm,right] {$\parentheses{4,8,2}$} (b);
	
	\path[->](a) edge node [xshift=-2mm,yshift=-1mm,left] {$\parentheses{5,7,6}$} (c);
	\path[->](a) edge node [xshift=0.7mm,yshift=-1mm,left] {$\parentheses{0,0,0}$} (d);
	\path[->](b) edge node [xshift=2.3mm,yshift=-1mm,right] {$\parentheses{0,0,0}$} (e);
	
	\path[->](c) edge node [left] {$\parentheses{0,0,0}$} (f);
	\path[->](d) edge node [left] {$\parentheses{3,1,8}$} (f);
	\path[->](d) edge node [right] {$\parentheses{0,0,0}$} (g);
	\path[->](e) edge node [right] {$\parentheses{0,0,0}$} (g);
	
	\path[->](f) edge node [left] {$\parentheses{0,0,0}$} (h);
	\path[->](g) edge node [left] {$\parentheses{0,0,0}$} (h);
	\path[->](g) edge node [right] {$\parentheses{4,5,4}$} (i);
	
	\path[->](h) edge node [left] {$\parentheses{2,3,6}$} (j);
	\path[->](h) edge node [right] {$\parentheses{0,0,0}$} (k);
	\path[->](i) edge node [right] {$\parentheses{0,0,0}$} (l);
	
	\path[->](j) edge node [right] {$\parentheses{1,3,5}$} (m);
	\path[->](k) edge node [right] {$\parentheses{1,3,5}$} (m);
	\path[->](k) edge node [right] {$\parentheses{0,0,0}$} (n);
	\path[->](l) edge node [right] {$\parentheses{0,0,0}$} (n);
	
	\path[->](m) edge node [right] {$\parentheses{0,0,0}$} (t);
	\path[->](n) edge node [right] {$\parentheses{2,8,2}$} (t);

	\end{tikzpicture}
	\caption{A valid network model for $\modo$ in Example~\ref{ex:modo}.}
	\label{fig:bddexample}
\end{figure}

\section{Network Model Construction}
\label{sec:NMConstructionAlgos}

\andre{
This section provides two frameworks for constructing valid network models for
MODOs.  The first approach relies on a recursive model of $\modo$.  The second is
a direct transformation from decision diagrams, which is
applicable when the objective functions are additively separable.
}

\subsection{Recursive Formulations}
\label{sec:recursiveModeling}

Several classes of single-objective optimization problems admit recursive
formulations, often written as DP models \citep{bertsimas1997introduction}.
These ideas were extended to the case of multiple objectives in the early work
by \cite{villarreal1982multicriteria}, which we build upon in this paper. In
particular, while DP models are intrinsically associated with a state-transition
graph, we show that multiobjective recursive formulations are analogously
associated with a valid network model.

\andre{
Formally, a multiobjective recursive formulation of a MODO $\modo$ is written in
terms of the following elements as depicted in Figure
\ref{fig:recursiongraph}: (i) $n+1$ state variables $s_0, s_1, \dots, s_n \in
\statespace$ for some state space $\statespace \subseteq \mathbb{R}^m$, where
the initial state $s_0$ is fixed; (ii) $n$ functions
$\feasiblevaluestate_1, \dots, \feasiblevaluestate_n : \statespace \rightarrow
2^{\mathbb{Z}}$ that
represent the variable-value assignments that can be applied at a state (i.e.,
they provide state-dependent feasible decisions); (iii) $n$ state transition
functions $\transition_1, \dots, \transition_n : \statespace \times \mathbb{Z}
\rightarrow \statespace$, each mapping a pair $(s,x)$ of a state $s \in
\statespace$ and a variable-value assignment $x \in \mathbb{Z}$ to another state
$s' \in \statespace$; and (iv) $n$ reward functions $\reward_1, \dots, \reward_n
: \statespace \times \mathbb{Z} \rightarrow \mathbb{R}^K$ that map an analogous
pair $(s,x)$ to a reward vector in $\mathbb{R}^K$. 
}

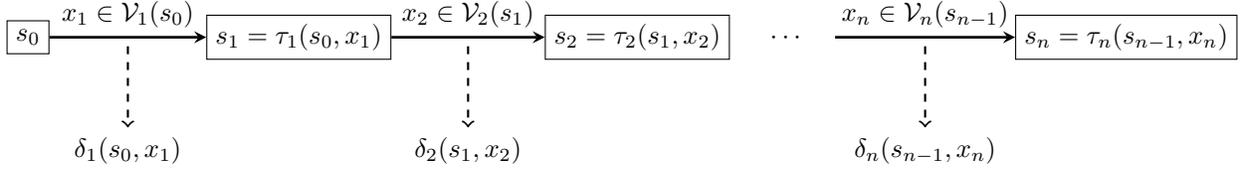
\begin{figure}[h]
	\centering
	\begin{tikzpicture}[scale=0.3][font=\sffamily, \small]
	
	\node [draw,rectangle] (s0) {$s_0$};
	\node [draw,rectangle,right of=s0,xshift=2.6cm] (s1) {$s_1 = \transition_1(s_0,x_1)$};
	\node [draw,rectangle,right of=s1,xshift=3.5cm] (s2) {$s_2 = \transition_2(s_1,x_2)$};
	\node [right of=s2,xshift=1cm] (dummy) {$\cdots$};
	\node [draw,rectangle,right of=dummy,xshift=3.5cm] (sn) {$s_n = \transition_n(s_{n-1},x_n)$};
	
	\path[->,>=stealth,line width=1pt] (s0) edge node[above] {$x_1 \in \feasiblevaluestate_1(s_0)$} (s1);
	\path[->,>=stealth,line width=1pt] (s1) edge node[above] {$x_2 \in \feasiblevaluestate_2(s_1)$} (s2);
	\path[->,>=stealth,line width=1pt] ([xshift=0.9cm] dummy.east) edge node[above] {$x_n \in \feasiblevaluestate_n(s_{n-1})$} (sn);
	
	\node [yshift=-1.5cm] (r1) at ($(s0.east)!0.5!(s1.west)$) {$\reward_1(s_0,x_1)$};
	\node [yshift=-1.5cm] (r2) at ($(s1.east)!0.5!(s2.west)$) {$\reward_2(s_1,x_2)$};
	\node [yshift=-1.5cm] (rn) at ($([xshift=0.9cm] dummy.east)!0.5!(sn.west)$) {$\reward_n(s_{n-1},x_n)$};
	
	\draw[->,dashed,thick] ([yshift=-0.4cm] s0 -| r1) -- (r1);
	\draw[->,dashed,thick] ([yshift=-0.4cm] s1 -| r2) -- (r2);
	\draw[->,dashed,thick] ([yshift=-0.4cm] s2 -| rn) -- (rn);
	
	\end{tikzpicture}
	\caption{Elements of a recursive formulation}
	\label{fig:recursiongraph}
\end{figure}
Based on these four components, a multiobjective recursive problem is of the form
\begin{align*}
\tag{$\modorec$}
\max_{s,x} 
\quad& 
\sum_{j=1}^n \reward_j(s_{j-1}, x_j) \\
\textnormal{s.t.}	\quad & s_j = \transition_j(s_{j-1}, x_j), &\textnormal{for all $j \in [n],$}\\
& x_j \in \feasiblevaluestate_j(s_{j-1}), &\textnormal{for all $j \in [n].$}
\end{align*}
In particular, each objective function evaluation is represented by the $k$-th index of the total reward tuple $\sum_{j=1}^{n} \reward_j(s_{j-1}, x_j)$ for $k \in [\nobj]$, i.e.,
\[
f_\objindex(s,x) := \left( \sum_{j=1}^n \reward_j(s_{j-1}, x_j) \right)_\objindex,
\]
while the feasible set of $\modorec$ is  
\andre{
\[
\xset := \left\{ (s,x) \in \statespace^{n+1} \times \mathbb{Z}^n \,:\, s_j = \transition_j(s_{j-1}, x_j), \ x_j \in \feasiblevaluestate_j(s_{j-1}) \;\; \ \text{for all} \ \,j \in [n] \right \}.
\]
}
Assume, without loss of generality, that $\transition_n(s,x) = s_\terminalnode$
for a fixed $s_\terminalnode$ and for all $s \in \statespace$, $x \in
\mathbb{Z}$. That is, the final transition always leads to a common
\textit{terminal} state $s_\terminalnode$, which can always be accomplished by
appropriately defining $\transition_n(\cdot)$. If a MODO can be written in the
form of $\modorec$, it exposes a recursive structure that can be immediately
leveraged for the construction of a network model. More specifically, the
network $\networkmodel = (\nodes, \arcs)$ is the state-transition graph defined
by the feasible set $\feasibleset$, which is composed as follows (see, e.g.,
\citealt{Bertsekas2017} for algorithmic details):

\begin{itemize}
	\item Node set $\nodes$: The nodes of $\networkmodel$ are associated with nodes of the state-transition graph. For any $j \in [n+1]$, there exists a node $u \in \layer{j}$ for every possible value of the state variable $s_{j-1}$ in $\feasibleset$. More specifically, at layer $j$ of $\networkmodel$, there exists one node for each state in the set $\text{Proj}_{s_{j-1}}(\feasibleset)$, defined as the projection of the feasible set $\feasibleset$ into the space of the variable $s_{j-1}$. With a slight abuse of notation, we let $\layer{j} = \{ u^j_1, \hdots, u^j_{|\layer{j}|} \} = \text{Proj}_{s_{j-1}}(\feasibleset)$. Note that since $s_0$ and $s_\terminalnode$ are fixed, layers $\layer{1}$ and $\layer{n+1}$ are singletons; namely $u^1_1 = s_0$ and $u^{n+1}_1 = s_\terminalnode$.
	\item Arc set $\arcs$: The arcs of $\networkmodel$ are associated with arcs of the state-transition graph. In other words, there exists an arc in $\networkmodel$ for every possible transition in the state-transition graph. Again with a slight abuse of notation, we represent an arc as a triplet, rather than a pair, $(\arcroot{a},\arcterminal{a},x)$ appending the variable-value information. Then, we have
	$$\arcs = \dot{\bigcup\limits_{j \in [n]}} \left\{ (u_i^j,u_{i'}^{j+1},x) : u_i^j \in \layer{j}, \ u_{i'}^{j+1} \in \layer{j+1}, \ x \in \feasiblevaluestate_j(u_i^j) \ \text{with} \ u_{i'}^{j+1} = \transition_j(u_i^j,x) \right\}.$$
	\item Arc weights: The reward functions provide the arc weights for $\networkmodel$. That is, if $a = (u_i^j,u_{i'}^{j+1},x) \in \arcs$, then 
$\arcweights{a} = \reward_{j}(u_i^j, x)$.
\end{itemize}  

It follows from the definition of arc weights that a path weight corresponds to the objective function in $\modorec$ and, by construction of the state-transition graph, that there is a one-to-one mapping between paths of $\networkmodel$ and solutions in $\feasibleset$. Thus, $\networkmodel$ is a valid network model for $\modorec$.

\begin{example}
\label{ex:modorec}

The instance in Example \ref{ex:modo} is a set packing instance, so it can be written as
\begin{align*}
\max \left \{ \left( (c^1)^\top x, (c^2)^\top x, \dots, (c^\nobj)^\top x \right) : Ax \le \boldsymbol{1}, x \in \mathbb{B}^n \right \}
\end{align*}
for an $m \times n$ matrix $A$ with elements $a_{ij} \in \mathbb{B}^n$, $i \in [m], j \in [n]$,
and cost vectors $c^j \in \mathbb{R}^n$, $j \in [n]$. 

A recursive formulation can be obtained 
as follows. A state variable $s \in \mathbb{B}^m$ is defined so that for any $i \in [m]$, $s_i = 1$ if and only if the $i$-th constraint of $Ax \le \boldsymbol{1}$ is satisfied as an equality or all the decision variables appearing in the constraint have already been assigned to zero. Therefore, we have the initial state $s_0 = \boldsymbol{0}$, whereas any final transition will lead to the terminal state $s_{\terminalnode} = \boldsymbol{1}$. Variable-value assignments $x_j \in \mathbb{B}$ represent the packing of element~$j$. We are not allowed to set a variable $x_j = 1$ if any constraint which includes $x_j$ holds as an equality; i.e., 
\[
\feasiblevaluestate_j(s) := \{ b \in \mathbb{B} \,:\, s_i + b \le 1 \ \text{for all} \ \,i \in [m] \ \text{with} \ a_{ij} = 1 \}.
\] 
The transition function ensures the consistency of the current state $s$ and the next state $\tilde{s}$ when setting $x_j = b$, $b \in \feasiblevaluestate_j(s)$. Specifically, let $m_i := \argmax \{ j \in [n] : a_{ij} = 1\}$ be the maximum index of the variables with a nonzero coefficient in the $i$-th constraint. Then, we have 
$\transition_j(s, b) = \tilde{s}$ where 
$$\tilde{s}_i = 
\left\{
		\begin{array}{lll}
			1, & \mbox{if } j = m_i \\
			s_i + b, & \mbox{if } j < m_i \ \text{and} \ a_{ij} = 1 \\
			s_i, & \mbox{otherwise}
		\end{array}
	\right.
\qquad \text{for all} \ i \in [m],
$$
for any $j \in [n]$. Finally, the reward function for each $j \in [n]$ is the contribution of $x_j$ to the objective function vector, which in this particular instance is state independent:
$\reward_j(s, b) := b \times \left( c^1_j, \, c^2_j, \, \dots, c^\nobj_j \right)$. 

For our particular instance in Example \ref{ex:modo}, this recursive formulation yields the network model depicted in Figure \ref{fig:recursivemodelexample} in the appendix. We note that this network model matches the one given in Figure \ref{fig:bddexample} until layer six (inclusive). The network model in Figure \ref{fig:bddexample} can be obtained from the one in Figure \ref{fig:recursivemodelexample} after applying Theorem \ref{thm:reduce} and Corollary \ref{cor:arcremoval} on layer seven.
\hfill $\square$
\end{example}

\subsection{Transformation from Decision Diagrams}
\label{sec:bddTransformation}

\andre{ 
A decision diagram (DD) is a network-based representation of a Boolean
or a discrete function with a large list of applications in mathematical
programming, operations research, and circuit design (see, e.g.,
\citealt{Bry92, Beh07,HadHoo07,BerHoeHoo11}). In such applications, DDs are used to
compactly represent or approximate the set of feasible solutions to a discrete
optimization problem. We refer to the work by \cite{bergman2016decision} for a
survey of concepts and existing DD methodologies for optimization. 
}

\andre{
Each layer in a DD $\mathcal{D}$, except the last one,
corresponds to a unique decision variable mapped via the bijective function $y:
[n] \rightarrow [n]$; that is, layer $l \in [n]$ corresponds to variable
$x_{y(l)}$. Let~$\Pathset_\mathcal{D}$ be the set of paths from root to
terminal in $\mathcal{D}$. The arc-domains associate each path $p = (a_1,
\ldots, a_n)$ in~$\Pathset_\mathcal{D}$ with a solution vector $x(p) \in
\mathbb{Z}^n$ defined by $x(p)_{y(\ell(a_j))} = \arcdomain{a_j}$ for all $j \in
[n]$. In particular, for a given feasible set $\xset \subset \mathbb{Z}^n$,
$\mathcal{D}$ is said to be \emph{exact} for $\xset$ if and only if $x(p) \in
\xset$ for all $p \in \Pathset_\mathcal{D}$ and $|\Pathset_\mathcal{D}| = |\xset|$. That
is, the paths of $\mathcal{D}$ encode $\xset$ exactly. \looseness=-1
}
\carlos{
In summary, decision diagrams are closely related to network models, but observe  
that arc-domains are exclusive to the former whereas arc-weight vectors it is are defined only in the latter.
}

Assume that, for a given MODO $\modo$, the objective function
$f_\objindex(\cdot)$ is separable for all $\objindex \in [\nobj]$; i.e.,
$$f_{\objindex}(x) := \sum_{j \in [n]} f_{k,j}(x_j)$$ for an appropriate choice
of functions $f_{k,j} : \mathbb{Z} \rightarrow \mathbb{R}$, $j \in [n]$. In
that case, we can transform an exact DD  \carlos{representing $\modo$ } into a network model $\networkmodel$
simply by defining suitable arc weights and removing the arc-domain labels.
More specifically, for every arc $a \in \arcs$, we define its weight according
to the variable corresponding to the layer of that arc, $x_{y(\ell(a))}$, and its
assigned value, $\arcdomain{a}$, as $\arcweights{a} :=
\parentheses{f_{1,y(\ell(a))}(\arcdomain{a}), \ldots,
f_{\nobj,y(\ell(a))}(\arcdomain{a})}.$ The
assignment of those weights to arcs implies that, for any path $p = (a_1,
\ldots, a_n)$ in the network, \[ w(p) = \sum_{j \in [n]}
\parentheses{f_{1,y(\ell(a_j))}(\arcdomain{a_j}), \ldots,
f_{\nobj,y(\ell(a_j))}(\arcdomain{a_j})} = f(x(p)), \] thus
$\paretofrontier{\networkmodel}$ is the Pareto frontier of $\modo$. This
relation was first used in \cite{BerCir16}.

\andre{
There exists an extensive number of DD construction algorithms for discrete
optimization problems with a (single) separable objective function
\citep{bergman2016decision}, some of which will be employed for our numerical
study in Section \ref{sec:copmutationalInsights}. Note, however, that our
proposed network models, after modifications by VPOs, do not necessarily map to
valid exact DDs. Indeed, our network operators are guaranteed only to preserve
the Pareto frontier, and as such the diagram is typically modified to readjust
weights, merge nodes, include infeasible paths, or exclude feasible paths to
allow for smaller network sizes (hence not satisfying the basic properties of
an exact DD).
}




As mentioned in Example \ref{ex:modorec}, our running Example \ref{ex:modo} is
a set packing instance. \S4 of \cite{bergman2016decision} provides a DD
representation for the set packing problem that can be readily transformed into
a network model by applying the methodology described above.

\section{Validity-preserving Operations}
\label{sec:VPOs}

This section is devoted to the study of generic algorithms, denoted by validity
preserving operations (VPOs), to transform a network $\networkmodel$ into a
smaller network $\networkmodel'$ such that $\paretofrontier{\networkmodel} =
\paretofrontier{\networkmodel'}$. For the structural results below, we assume
that an initial valid network model $\networkmodel = (\nodes, \arcs)$ is
available.

\subsection{Weight Shifting and Node Merging} 
\label{sec:reduction}


We extend the classical concept of \textit{reduction} proposed by \cite{Bry86}
as a VPO for network models. Reduction is an operation applied to DDs that
merge nodes which share isomorphic subgraphs. \cite{Hooker2013} later extended
this notion for DPs with state-dependent rewards, showing that further merging
can be achieved for single-objective problems by considering isomorphism with
respect to arc weights.

Based on this concept, we define a weight-shifting operation that, given a node $u$, transfers weights from the arcs directed out of $u$ (i.e., $\arcs^+(u)$) to those directed into $u$ (i.e., $\arcs^-(u)$). 
 
\begin{proposition}[\bf Weight shift]
	\label{prop:shift}
	Let $u \in \nodes$ be a node in $\networkmodel$ such that $u  \notin \set{\rootnode, \terminalnode}$. For any $\tilde{c} \in \mathbb{R}^\nobj$, 
	the operations defined by (1) and (2) below is a VPO:
	\begin{enumerate}
		\item $\arcweights{a} := \arcweights{a} - \tilde{c}$ for all $a \in A^+(u)$; and
		\item $\arcweights{a} := \arcweights{a} + \tilde{c}$ for all $a \in A^-(u)$.
	\end{enumerate}
\end{proposition}

	\begin{proof}
	 Arc-weight $w(p)$ remains unchanged for all $p \in \Pathset$ and thus $\paretofrontier{\networkmodel}$ is unchanged.
	\end{proof}
	
	\begin{example}
	
  Consider the network model in Figure~\ref{fig:bddexample}. Let $\tilde{c} =
  (2,8,2)$. The network model resulting from subtracting $\tilde{c}$ from
  arc-weight $\arcweights{u_2^7,\terminalnode}$ and adding $\tilde{c}$ to arcs
  $\arcweights{u_2^6,u_2^7}$ and $\arcweights{u_3^6,u_2^7}$ corresponds to the 
  operations in Proposition \ref{prop:shift}. The last
  three layers of $\networkmodel$ are reproduced in
  Figure~\ref{fig:costShift1}. The result of applying the weight shift is
  depicted in Figure~\ref{fig:costShift2}.
		\hfill $\square$
	\end{example}
	
\begin{figure}[h]
	\centering
	\begin{minipage}[t]{0.3\textwidth}
		\centering
	\begin{tikzpicture}[scale=0.3][font=\sffamily,\tiny]

	\node [draw,circle,scale=1] (j) at (-6,-29) {$u^6_1$};
	
	\node [draw,circle,scale=1] (k) at (0,-29) {$u^6_2$};
	
	\node [draw,circle,scale=1] (l) at (6,-29) {$u^6_3$};
	
	\node [draw,circle,scale=1] (m) at (-3,-35) {$u^7_1$};
	
	\node [draw,circle,scale=1] (n) at (3,-35) {$u^7_2$};
	
	\node [draw,circle,scale=1] (t) at (0,-41) {$\terminalnode$};
	

	\path[->](j) edge node [xshift=-2mm,yshift=2mm,left] {$\parentheses{1,3,5}$} (m);
	\path[->](k) edge node [xshift=1mm,yshift=2mm,left] {$\parentheses{1,3,5}$} (m);
	\path[->](k) edge node [xshift=-1mm,yshift=2mm,right] {$\parentheses{0,0,0}$} (n);
	\path[->](l) edge node [xshift=2mm,yshift=2mm,right] {$\parentheses{0,0,0}$} (n);
	
	\path[->](m) edge node [left] {$\parentheses{0,0,0}$} (t);
	\path[->](n) edge node [right] {$\parentheses{2,8,2}$} (t);

		\end{tikzpicture}
		\caption{Original arc-weights}
		\label{fig:costShift1}
	\end{minipage}
	\hfill
	\begin{minipage}[t]{0.3\textwidth}
				\centering
	\begin{tikzpicture}[scale=0.3][font=\sffamily,\tiny]

	\node [draw,circle,scale=1] (j) at (-6,-29) {$u^6_1$};
	
	\node [draw,circle,scale=1] (k) at (0,-29) {$u^6_2$};
	
	\node [draw,circle,scale=1] (l) at (6,-29) {$u^6_3$};
	
	\node [draw,circle,scale=1] (m) at (-3,-35) {$u^7_1$};
	
	\node [draw,circle,scale=1] (n) at (3,-35) {$u^7_2$};
	
	\node [draw,circle,scale=1] (t) at (0,-41) {$\terminalnode$};
	

	\path[->](j) edge node [xshift=-2mm,yshift=2mm,left] {$\parentheses{1,3,5}$} (m);
	\path[->](k) edge node [xshift=1mm,yshift=2mm,left] {$\parentheses{1,3,5}$} (m);
	\path[->](k) edge node [xshift=-1mm,yshift=2mm,right] {$\parentheses{2,8,2}$} (n);
	\path[->](l) edge node [xshift=2mm,yshift=2mm,right] {$\parentheses{2,8,2}$} (n);
	
	\path[->](m) edge node [left] {$\parentheses{0,0,0}$} (t);
	\path[->](n) edge node [right] {$\parentheses{0,0,0}$} (t);

	\end{tikzpicture}
	\caption{Weight shift by $\tilde{c} = (2,8,2)$ at node $u_2^7$}
	\label{fig:costShift2}
	\end{minipage}
	\hfill
	\begin{minipage}[t]{0.3\textwidth}
		\centering
	\begin{tikzpicture}[scale=0.3][font=\sffamily,\tiny]

	\node [draw,circle,scale=1] (j) at (-6,-29) {$u^6_1$};
	
	\node [draw,circle,scale=1] (k) at (0,-29) {$u^6_2$};
	
	\node [draw,circle,scale=1] (l) at (6,-29) {$u^6_3$};
	
	\node [draw,circle,scale=1] (o) at (0,-35) {$u^7_1$};
	
	\node [draw,circle,scale=1] (t) at (0,-41) {$\terminalnode$};
	

	\path[->](j) edge node [xshift=-2mm,yshift=2mm,left] {$\parentheses{1,3,5}$} (o);
	\path[->](k) edge [bend right] node [xshift=1mm,yshift=4mm,left] {$\parentheses{1,3,5}$} (o);
	\path[->](k) edge [bend left] node [xshift=-1mm,yshift=4mm,right] {$\parentheses{2,8,2}$} (o);
	\path[->](l) edge node [xshift=2mm,yshift=2mm,right] {$\parentheses{2,8,2}$} (o);
	
	\path[->](o) edge node [right] {$\parentheses{0,0,0}$} (t);

	\end{tikzpicture}
	\caption{Result of merge operation on nodes~$u_1^7$ and~$u_2^7$}
	\label{fig:costShift3}
	\end{minipage}
\end{figure}

We now present a sufficient
condition for when merging nodes is a VPO.
	
\begin{theorem}[Node merge]
	\label{thm:reduce}
	Let $u_1,u_2 \in \nodes$, $u_1 \neq u_2$, be two nodes in $\networkmodel$ for which there exists a one-to-one mapping between the arcs in $\arcs^+(u_1)$ and $\arcs^+(u_2)$ satisfying that, for every pair of arcs $(a_1,a_2)$ such that $a_1 \in \arcs^+(u_1)$ maps to $a_2 \in\arcs^+(u_2)$, we have $\arcterminal{a_1} = \arcterminal{a_2}$ and $w(a_1) = w(a_2)$.  Then, the following sequence of operations characterizes a VPO:
	\begin{enumerate}
		\item Delete all arcs in $\arcs^+(u_2)$;
		\item Redefine $\arcterminal{a} = u_1$ for all $a \in \arcs^-(u_2)$; and
		\item Delete $u_2$,
	\end{enumerate}
\end{theorem}

\begin{example}
	Figure~\ref{fig:costShift3} depicts the result of merging nodes $u_1^7$ with $u_2^7$ in Figure~\ref{fig:costShift2}.  Note that the resulting network model has fewer arcs and nodes, but the same number of paths.  Such an operation thereby decreases the size of the network model without altering  $\paretofrontier{\networkmodel}$.
	\hfill $\square$
\end{example}

Proposition~\ref{prop:shift} and Theorem~\ref{thm:reduce} define a strategy for
simplifying a network model. In particular, starting from the penultimate layer
and moving upwards in the network, we process each node in the layer in the
following way. We construct the vector~$\tilde{c}(u)$ for each node~$u$ in the
inspected layer by taking the component-wise minimum arc-weight among arcs in
$\arcs^{+}(u)$. Once~$\tilde{c}(u)$ has been obtained, the arc-weights are
shifted up as prescribed in Proposition~\ref{prop:shift}. After repeating
this operation to all nodes in a layer, the conditions of
Theorem~\ref{thm:reduce} and its VPO are applied to the nodes
of the same layer in the transformed network. Node-merge operations
can be performed in any order, in the sense that any sequence will lead to the
same reduced network.

\begin{proposition}
\label{prop:VPOcomplexity}
	The component-wise minimum arc-weight shift and node merge VPO 
	have worst-case time complexity of $\bigO(\nobj|\nodes|^2\log(|\nodes|))$.
\end{proposition}

\subsection{Arc removal}
\label{sec:arcRemoval}

In this section, we investigate algorithms and structural results of VPOs that eliminate arcs of the network model. 
For an arc $a \in \arcs$, let $\networkmodel - a$ be the network model resulting from the removal of $a$ from $\networkmodel$.  By definition, removing arc $a$ is a VPO if and only if $\paretofrontier{\networkmodel} = \paretofrontier{\networkmodel - a}$.  The following theorem shows that identifying when such condition holds in general is an NP-Hard problem.

\begin{theorem}
	\label{thm:NPHardRemoval}
	Given a valid network model $\networkmodel = (\nodes, \arcs)$ for a MODO $\modo$, let
	$\tilde{\arcs} \subseteq \arcs$ be a subset of arcs of $\networkmodel$. Deciding whether there exists an arc $a \in \tilde{\arcs}$
	such that $\paretofrontier{\networkmodel} = \paretofrontier{\networkmodel - a}$
	is NP-hard. 
\end{theorem}


Despite the hardness of determining whether an arc can be safely removed without changing the Pareto frontier, we can still exploit strong sufficient conditions for designing arc-removal VPOs. Additional notation is in order.  Given a network model $\networkmodel = (\nodes, \arcs)$ and 
two nodes $u,v \in \nodes$ such that 
$\nodelayer{u} < \nodelayer{v}$, let $\networkmodel[u,v]$ be the network model containing only nodes and arcs in $\networkmodel$ that lie on some path which starts at node $u$ and ends at node $v$.  
We introduce the following concept:
\begin{definition}
	\label{def:isolating}
A pair of nodes~$u,v \in \nodes$ is \emph{isolating} in~$\networkmodel$ when, for every arc $a \in \arcs_{\networkmodel} \backslash \arcs_{\networkmodel[u,v]}$,
\begin{itemize}
	\item[(i)] $\arcterminal{a} \in \nodes_{\networkmodel[u,v]}$ implies that $\arcterminal{a} = u$; and
	\item[(ii)] $\arcroot{a} \in \nodes_{\networkmodel[u,v]}$ implies that $\arcroot{a} = v$.
\end{itemize}
\end{definition}

According to Definition \ref{def:isolating}, nodes $u$ and $v$ are isolating in~$\networkmodel$ if $\networkmodel[u,v]$ contains all arcs from~$\networkmodel$  that are directed to nodes in  $\nodes_\networkmodel[u,v] \setminus \{u\}$ and all arcs that are directed out of nodes in $\nodes_\networkmodel[u,v]  \setminus \{v\} $. 
(For example, any pair of nodes are isolating in Figure \ref{fig:UnconstraintMODO}\subref{fig:unconstNetwork}.)  Note that one can check whether two nodes $u$ and $v$ are isolating in~$\networkmodel$ in polynomial time, in the size~$\networkmodel$, with a breadth-first search. 

Pairs of isolating vertices yield a sufficient condition for an arc-removal operation to be a VPO.

\begin{theorem}
	\label{thm:subBDDarcRemoval}
	Let $u$ and $v$ be isolating nodes in a network model $\networkmodel$.  For any $a \in \arcs_{\networkmodel[u,v]}$, if  $\paretofrontier{\networkmodel[u,v]} = \paretofrontier{\networkmodel[u,v] - a}$, then  $\paretofrontier{\networkmodel} = \paretofrontier{\networkmodel - a}$. That is, if the removal of arc $a$ is a VPO in $\networkmodel[u,v]$, then it is also a VPO in $\networkmodel$.
\end{theorem}

The proof of Theorem \ref{thm:subBDDarcRemoval} is provided in Section \ref{subsec:ProofThmThree}.

\smallskip

Theorem~\ref{thm:subBDDarcRemoval} shows that
 pairs of isolating nodes in an arbitrary network model~$\networkmodel$ define  subnetworks whose VPOs involving the removal of arcs are also VPOs for~$\networkmodel$.
The simplest case reduces to two arcs with the same endpoints, which yields the following immediate corollary. 
\begin{corollary}
\label{cor:arcremoval}
	Let $a_1$ and $a_2$ be any two arcs of a network model~$\networkmodel$ for which $\arcroot{a_1} = \arcroot{a_2}$ and $\arcterminal{a_1} = \arcterminal{a_2}$.  If $\arcweights{a_1} \strictlydominated \arcweights{a_2}$ or $\arcweights{a_1} = \arcweights{a_2}$, then the removal of $a_1$ is a VPO.
\end{corollary}

Theorem~\ref{thm:subBDDarcRemoval} can also be applied directly by choosing two
nodes $u,v$ such that $N[u,v]$ is sufficiently small so that all associated
arcs can be removed efficiently. Specifically, given $\Delta := \nodelayer{v} -
\nodelayer{u}$, the number of paths in $\networkmodel[u,v]$ is bounded by
$\bigO(2^\Delta)$, and hence for small $\Delta$ the Pareto frontier (and arc to
be removed) can be identified quickly using, e.g., the procedures from Section
\ref{sec:findingTheParetoFrontier}, which we discuss next. For our numerical
evaluation, we fixed $\Delta = 2$. That is, (i) we find a pair $(u,v)$ of
isolating nodes that are distant by at most two layers, (ii) obtain the network
model $\networkmodel[u,v]$, (iii) compute its Pareto frontier, and finally (iv)
apply Theorem~\ref{thm:subBDDarcRemoval} to remove arcs that are VPOs in
$\networkmodel$.

\section{Generating the Pareto Frontiers from a Network Model}
\label{sec:findingTheParetoFrontier}

Given a valid network model, finding the Pareto frontier generally reduces to solving an
MSP (by multiplying arc-weights by $-1$) in a layered-acyclic multi-digraph, for which an
extensive literature exists; see, e.g., surveys by \citealt{Tarapata2007} and
\citealt{GarGioTav10}.

\andre{
In this section, we propose two methodologies for
enumerating the Pareto frontier based on our network model structure. The first
is a direct modification of the unidirectional recursion by \cite{Hen86}, also
applied, e.g., in \cite{figueira2013algorithmic} and \cite{rong2014dynamic}.
The second technique is an extension of \cite{GalandIPS13} and performs a
bidirectional search that combines the partial Pareto frontiers of each layer
using a \emph{coupling} operator. Both methodologies assign a set (or a collection
of sets) of $\nobj$-dimensional vectors to nodes of the network. Each
$\nobj$-dimensional vector is henceforth referred to as a \emph{label}, as is
done in the MSP literature.
}

\subsection{Unidirectional Pareto frontier generation}
\label{sec:unidirectional}

The unidirectional algorithms process one layer at a time, computing the
partial nondominated solutions at a node based on either the incoming arcs or
the outgoing arcs. It is a direct application of the recursion by \cite{Hen86}
but using the underlying structure of the network model, similar to the version
presented by \cite{rong2014dynamic}.

The procedure works as follows. When processed from the root node~$\rootnode$
to the terminal node~$\terminalnode$, the algorithm assigns a single set of
\emph{top-down labels} $\labelSetTD{u}$ to each node $u$. The label set of each
node is initialized as the empty set, except the root node, where
$\labelSetTD{\rootnode}$ is initialized as $\{\textbf{0}\}$. For each layer $j$
from one to $n$, having constructed $\labelSetTD{u}$ for all $u \in \layer{j}$,
the labels are constructed for the nodes in $\layer{j+1}$ by considering the
arcs directed from nodes in $\layer{j}$ to $\layer{j+1}$, one by one. For each
such arc $a$ (i.e., $a \in \bigcup_{u \in \layer{j}} \arcs^+(u)$) and every
label $z \in \labelSetTD{\arcroot{a}}$, the label $z + w(a)$ is added to
$\labelSetTD{\arcterminal{a}}$. After all arcs directed out of nodes in
$\layer{j}$ are processed, $\labelSetTD{u}$ is re-assigned to
$\nd{\labelSetTD{u}}$, to remove any labels that are dominated by other labels
in the set. Note that one can also do a simple check each time a label is added
to see if it is dominated by another label already existing for the node. At
the culmination of the algorithm, $\labelSetTD{\terminalnode}$ will be
$\paretofrontier{\networkmodel}$.

One can also run the algorithm in the opposite direction, starting from $\terminalnode$ and flipping the direction of the arcs.  The terminal node is initialized as $\{\textbf{0}\}$, and the nodes are processed in the opposite direction. We refer to the labels constructed in this direction as \emph{bottom-up labels} $\labelSetBU{u}$.  The set of labels $\labelSetBU{\rootnode}$ coincides with $\labelSetTD{\terminalnode}$ and is therefore equal to $\paretofrontier{\networkmodel}$.

\begin{example}
\label{ex:unidirectional}
Consider Figure~\ref{fig:bddexample}. The labels on the left of the nodes (shown in red) correspond to the top-down labels (i.e., for every node $u$, they list $\labelSetTD{u}$).  A ``$\checkmark$" is drawn next to labels that remains in $\labelSetTD{u}$ after the application of the $\nd{}$ operator, and a ``$\times$" is drawn otherwise. 
To the right of each node $u$, $\labelSetBU{u}$ is listed (in blue), with symbols ``$\checkmark$" and ``$\times$" indicating whether the labels remain or are discarded after the application of the $\nd{}$ operator, respectively.
\hfill $\square$
\end{example}

\subsection{Bidirectional Pareto frontier generation}
\label{sec:bidirectional}

We now provide a compilation method that extends the work of
\cite{GalandIPS13} for network models. Namely, one may obtain the elements composing the Pareto frontier by constructing labels in both directions simultaneously and \emph{coupling} the top-down and bottom-up labels.
 Given two sets of vectors $\mathcal{Z}_1, \mathcal{Z}_2 \subseteq \mathbb{R}^\nobj$, define the coupling of $\mathcal{Z}_1$ and $\mathcal{Z}_2$ as
 \[ 
 \couple{\mathcal{Z}_1}{\mathcal{Z}_2} := \nd{\set{z: z = z^1 + z^2, z^1 \in \mathcal{Z}_1, z^2 \in \mathcal{Z}_2}}.
 \]
The coupling function~$ \couple{\mathcal{Z}_1}{\mathcal{Z}_2}$ returns the nondominated set of vectors that result from every pairwise sum of vectors from the two sets~$\mathcal{Z}_1$ and~$\mathcal{Z}_2$. 
 
Let us fix a layer $j'$ and suppose we created the labels $\labelSetTD{u}$ for every node $u \in \layer{j}$, $j \leq j'$, and the labels $\labelSetBU{u}$ for every node $u \in \layer{j}$, $j \geq j'$.  We define the operation of \emph{coupling on layer} $j'$ as
  \[
  \coupleLayer{\layer{j'}} := \nd{ \bigcup_{u \in \layer{j'}} \couple{\labelSetTD{u}}{\labelSetBU{u}} }. 
  \] 
This yields the nondominated set that results from the coupling of the top-down and the bottom-up labels on each node. Note that $\coupleLayer{\layer{j}} = \nds$ for any
$j \in [n+1]$.

%
Example~\ref{ex:coupling} shows that this approach can significantly reduce the number of operations required to find the Pareto frontier of a network model.
Since the nondominated frontier of any set~$\mathcal{S}$ of $\nobj$-dimensional vectors can be found in time $\mathcal{O}\parentheses{|\mathcal{S}|\cdot \parentheses{\log\parentheses{|\mathcal{S}|}}^{\nobj - 2}}$
\citep{borzsony2001skyline}, the coupling operation of sets~$Z_1$ and~$Z_2$ can be completed in time $\bigO\parentheses{|\mathcal{Z}_1|\cdot|\mathcal{Z}_2|\cdot \parentheses{\log\parentheses{|\mathcal{Z}_1|\cdot|\mathcal{Z}_2|}}^{\nobj - 2}}$. 

\begin{example}
	\label{ex:coupling}
	Consider the network model in Figure~\ref{fig:bddexample}.  Suppose we fix $\layer{5}$, composed of nodes $u_1^5$ and $u_2^5$, as the coupling layer. Only 14 top-down labels need to be created to find $\labelSetTD{u_1^5}$ and $\labelSetTD{u_2^5}$, and only 11 bottom-up labels need to be created in order to find $\labelSetBU{u_1^5}$ and  $\labelSetBU{u_2^5}$.   The coupling of these sets results in
	\begin{eqnarray*}
	\couple{\labelSetTD{u_1^5}}{\labelSetBU{u_1^5}} & = & \set{\parentheses{8,13,17},\parentheses{6,7,19},\parentheses{7,14,13},\parentheses{7,15,8},\parentheses{6,16,4}} \\
	\couple{\labelSetTD{u_2^5}}{\labelSetBU{u_2^5}} & = & \set{\parentheses{10,21,8}},
	\end{eqnarray*}
	and, finally, 
	$\coupleLayer{\layer{5}} = 	\set{
		\parentheses{8,13,17},
		\parentheses{6,7,19},
		\parentheses{7,14,13},
		\parentheses{10,21,8}
	},$
	as desired. Note that using either unidirectional approach requires the creation of a total of 36 labels, as opposed to the 25 required using coupling. \hfill $\square$ 
\end{example}

Determining the best layer to couple on (i.e., the one in which the number of
labels that need to be created is minimized) is a non-trivial. In particular, a
unidirectional Pareto frontier compilation may require the creation of fewer
labels than the bidirectional variant. We therefore propose the following
heuristic procedure. Starting from $\rootnode$ and $\terminalnode$, we first
create $\labelSetTD{u'}$ and $\labelSetBU{u''}$, respectively, for all $u' \in
\layer{2}$ and for all $u'' \in \layer{n}$. Then, having constructed top-down
labels for each node up to $\layer{j_1}$, $2 \leq j_1$, and bottom-up labels
for each node on or below $\layer{j_2}$, $j_2 \in [j_1+1,n]$, we pick
among~$j_1$ and $j_2 $ the layer with the fewer number of total labels in order
to proceed with the extension operations. Namely, if $\sum_{u \in
\layer{j_1}}|\labelSetTD{u}| \leq \sum_{u \in \layer{j_2}}|\labelSetBU{u}|$,
extension of the top-down labels to $\layer{j_1 + 1}$ is completed and $j_1$ is
set to $j_1 + 1$. Otherwise, extension of the bottom-up labels to $\layer{j_2 -
1}$ is completed, and $j_2$ is set to $j_2 - 1$. This procedure is repeated
until $j_1 = j_2$, upon which coupling on layer~$\layer{j_1}$ is used to
calculate the Pareto frontier of the network model.



\subsection{Label Removal Algorithms}
\label{sec:labelRemovalAlgos}

Given a network model, the complexity of finding the Pareto frontier is largely determined by the cardinality of $\labelSetTD{u}$ and $\labelSetBU{u}$. Thus, having a VPO for the reduction of  $|\labelSetTD{u}|$ and $|\labelSetBU{u}|$ is relevant for computational purposes. 

The following proposition introduces such an operation. Given two nodes $u,v$ in a same layer
$\layer{j}$, the intuition behind the proposition is to identify when the subnetwork associated with $\networkmodel[u,\terminalnode]$ is dominated by $\networkmodel[v,\terminalnode]$, in which case we can remove labels in $u$ that are dominated by $v$. 

\begin{proposition}[Label filtering]
	\label{prop:labelFiltering}
	Let $u$ and $v$ be two nodes in $\layer{j}$ for some $j \in [n]$.  Suppose
	\[
	\nd{
		\paretofrontier{\networkmodel[u,\terminalnode]}
		,
		\paretofrontier{\networkmodel[v,\terminalnode]}
	}
	=
	\paretofrontier{\networkmodel[v,\terminalnode]}.
	\]
	If there exists a pair of labels $\ell^u \in \labelSetTD{u}$ and $\ell^v \in \labelSetTD{v}$ for which $\ell^u \strictlydominated \ell^v$ (or $\ell^u = \ell^v$), then removing $\ell^u$ from $\labelSetTD{u}$ is a VPO. Similarly, if
	\[
	\nd{
		\paretofrontier{\networkmodel[\rootnode,u]}
		,
		\paretofrontier{\networkmodel[\rootnode,v]}
	}
	=
	\paretofrontier{\networkmodel[\rootnode,v]},
	\]
	and there exists $\ell^u \in \labelSetBU{u}$ and $\ell^v \in \labelSetBU{v}$ for which $\ell^u \strictlydominated \ell^v$ (or $\ell^u = \ell^v$), 
	then removing $\ell^u$ from $\labelSetBU{u}$ is a VPO.
\end{proposition}

\begin{proof}
We provide a proof of the first case, 
as the proof of the other case follows by inverting the network model. By the condition in the statement of the proof, for each path $p$ in $\networkmodel_{[u,\terminalnode]}$, it must be that $w(p) + \ell^u$ is dominated by $w(p') + \ell^v$, for some path $p'$ in $\networkmodel_{[v,\terminalnode]}$.  Thus $w(p) + \ell^u$ does not belong to~$\paretofrontier{\networkmodel}$ for any such~$p$. It follows that  the removal of $\ell^u$ is a VPO. 
\end{proof}

Proposition \ref{prop:labelFiltering} generalizes the concept of \textit{state-based dominance} in DP \citep{Bertsekas2017} to network models. In particular, we can incorporate domain-specific information into a network model to identify cases where the conditions of Proposition \ref{prop:labelFiltering} are satisfied. We provide an example instantiation in Section \ref{sec:copmutationalInsights} which results in significant speedups in enumerating the Pareto frontier.

\section{Numerical Study}
\label{sec:copmutationalInsights}

In this section, we provide a detailed numerical evaluation of the
effectiveness of the proposed algorithm on five different classes of problems.
For each class, we discuss how the initial network model is constructed,
explain the best algorithmic configuration, and compare with existing
state-of-the-art approaches for general MODOs. In particular, we consider the
methodologies proposed by \cite{kirlik2014new} and \cite{Ozlen2013}, hereafter
denoted by algorithms \kirlik{} and \aira{}, respectively. The source codes of
these algorithms were kindly provided by the respective authors.


\paragraph{Experimental Design and Evaluation.}
All experiments ran on an Intel(R) Xeon(R) CPU E5-2680 v2 at 2.80GHz. Each experiment was limited to one thread and subject to a time limit of 3,600 seconds and a memory limit of 16GB. 
The algorithms \kirlik{} and \aira{}  depend on the resolution of integer linear programs; we employed IBM ILOG CPLEX 12.7.1 with default settings for this task \citep{CPLEXRef}.  The paired Wilcoxon signed-rank test was used to estimate $p$-values comparing pairs of algorithms \citep{wilcoxon1945individual} (this is a nonparametric test that will be used to compare if population mean ranks of solution times differ between algorithms).


Data was generated following previous literature guidelines; details of the
random generation procedure for each problem class are presented in the appendix. Direct comparisons between the algorithms are presented in
cumulative distribution plots, which show the number of instances solved by
each algorithm ($y$-axis) within a given time limit ($x$-axis). We use the
integral of the curve associated with an algorithm to estimate its relative
performance so that the quality of an algorithm depends both on the running
time and on the number of instances solved. The order in which algorithms are
listed in the legends reflect this metric, with the best-performing algorithms
appearing on the top. Additionally, we present scatter plots to compare the
best-performing network model algorithm against the best previous
state-of-the-art algorithm. These plots are presented in logarithmic scale and
represent the amount of time the algorithms require to solve each instance of
the given benchmark. We also employ a color code to indicate the size of the
Pareto frontier of each instance.

The network model-based algorithms employing the bottom-up, top-down, and bidirectional Pareto frontier compilation are represented by~\bu{}, \td{}, and~\coup{}, respectively. For applications where the domain-specific label filtering given in Proposition \ref{prop:labelFiltering} has been applied, we denote the extensions of \bu{}, \td{}, and \coup{} by \bup{}, \tdp{}, and \coupp{}, respectively.

\subsection{Multiobjective 0-1 Knapsack Problem} \label{sec:knapsack}

Given $n$ items, a capacity $W > 0$, and for each item $i \in [n]$, a weight $w_i > 0$ and 
$\nobj$ profits $p^1_i, p^2_i, \dots, p^\nobj_i > 0$, the multiobjective 0-1 knapsack problem (MKP) is
\[
\max \left \{ ((p^1)^\top x, (p^2)^\top x, \dots, (p^\nobj)^\top x) \,:\, 
 w^\top x \le W, \; x \in \mathbb{B}^n \right \}.\]

\paragraph{Network model construction.} The initial network is constructed via
a recursive formulation using a single dimensional state variable $s \in \mathbb{R}_+$, which
corresponds to the total weight of the knapsack at a certain stage. The root
state is $s_0 = 0$. We cannot set a variable $x_j = 1$ if it weighs more than
the available capacity, i.e., $\feasiblevaluestate_j(s) := \{ v \in \mathbb{B}
\,:\, s + v \cdot w_j \leq W \}$. The transition functions update the total
weight of the knapsack: $\transition_j(s, v) = s + v \cdot w_j$ for all $j \in
[n-1]$ and $\transition_n(s, v) = W$. Lastly, for any $j \in [n]$, the reward
function is the profit vector of item $j$, i.e., $\reward_j(s, v) := v \times
\left( p^1_j, \, p^2_j, \, \dots, p^\nobj_j \right)$.

We incorporate the label filtering of Proposition~\ref{prop:labelFiltering} by exploiting
the classical DP state dominance for knapsack problems. For any $j \in [n]$, let $u^j_i, u^j_{i'} \in \layer{j}$ be two possible states obtained at layer $j$ using the aforementioned recursive model. If $u^j_i \geq u^j_{i'}$, we have  
	\[
	\nd{
		\paretofrontier{\networkmodel[u^j_{i},\terminalnode]}
		,
		\paretofrontier{\networkmodel[u^j_{i'},\terminalnode]}
	}
	=
	\paretofrontier{\networkmodel[u^j_{i'},\terminalnode]}.
	\]
The equality above follows since each path (i.e., partial feasible solution) in $\networkmodel[u^j,\terminalnode]$ has a path in $\networkmodel[u^j_{i'},\terminalnode]$ of same path-weight, given that we have more capacity available at $u^j_{i'}$ than $u^j_i$. That is, we can assign the same variable values in a path starting at $u^j_{i'}$ and incur the same objective function contribution. Thus, we can remove labels at $u^j_i$ if they are dominated by a label at $u^j_{i'}$.

\paragraph{Computational evaluation.}

We experimented on 450 instances with $K \in \{3,4,\dots,7\}$ objectives and $n \in
\{20,30,\dots,100\}$ variables. The detailed results are presented in
Table~\ref{tab:KnapsackResults} of the appendix. We provide an
analysis of the results and summarize our findings.

\begin{figure}[h]
	\centering
	\begin{minipage}[b]{0.5\textwidth}
		\includegraphics[width=0.94\textwidth]{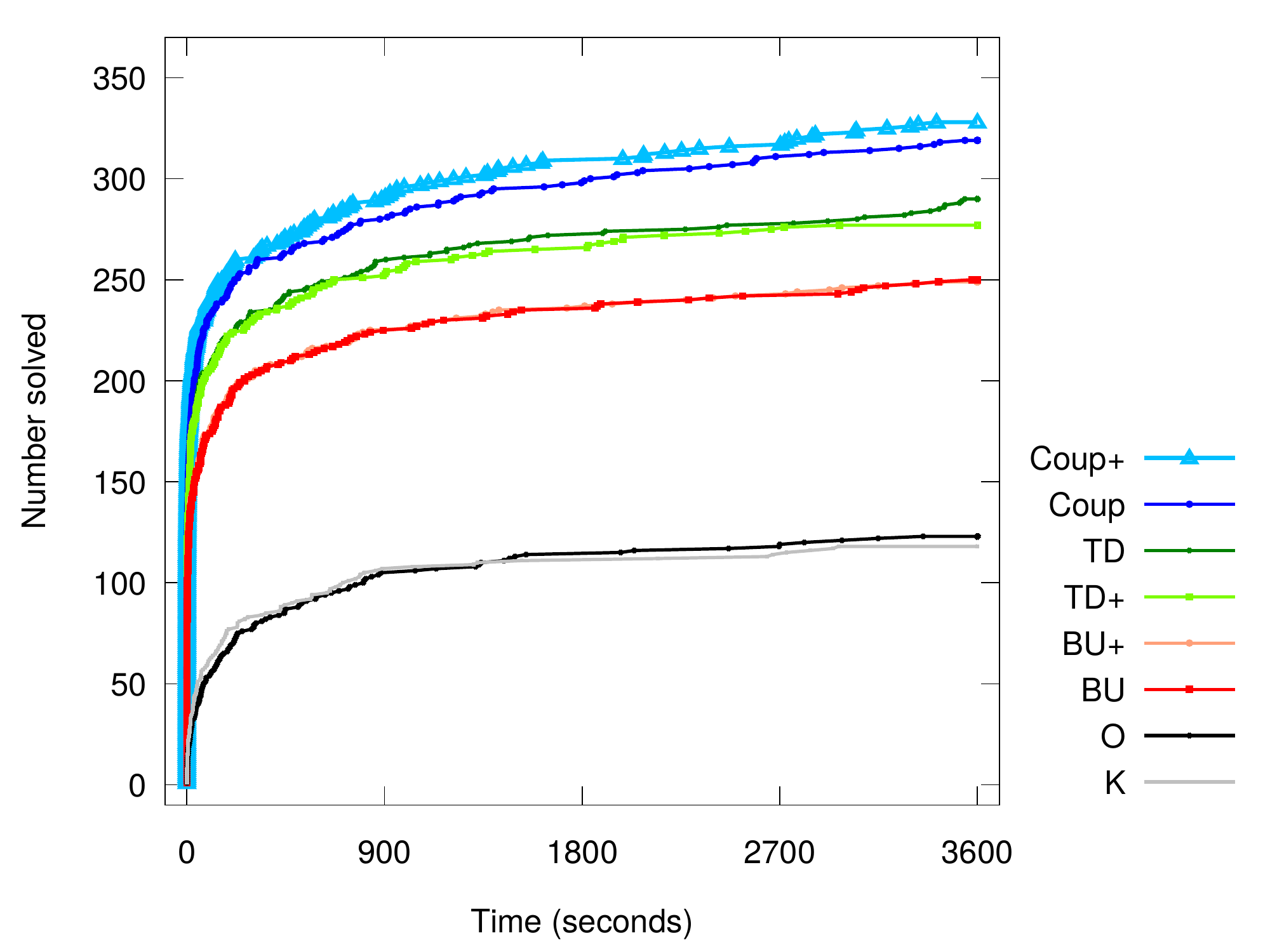}
		\caption{Knapsack cumulative distribution plot}
		\label{fig:knapCD}
	\end{minipage}
	\hfill
	\begin{minipage}[b]{0.45\textwidth}
		\centering
		\includegraphics[width=0.98\textwidth]{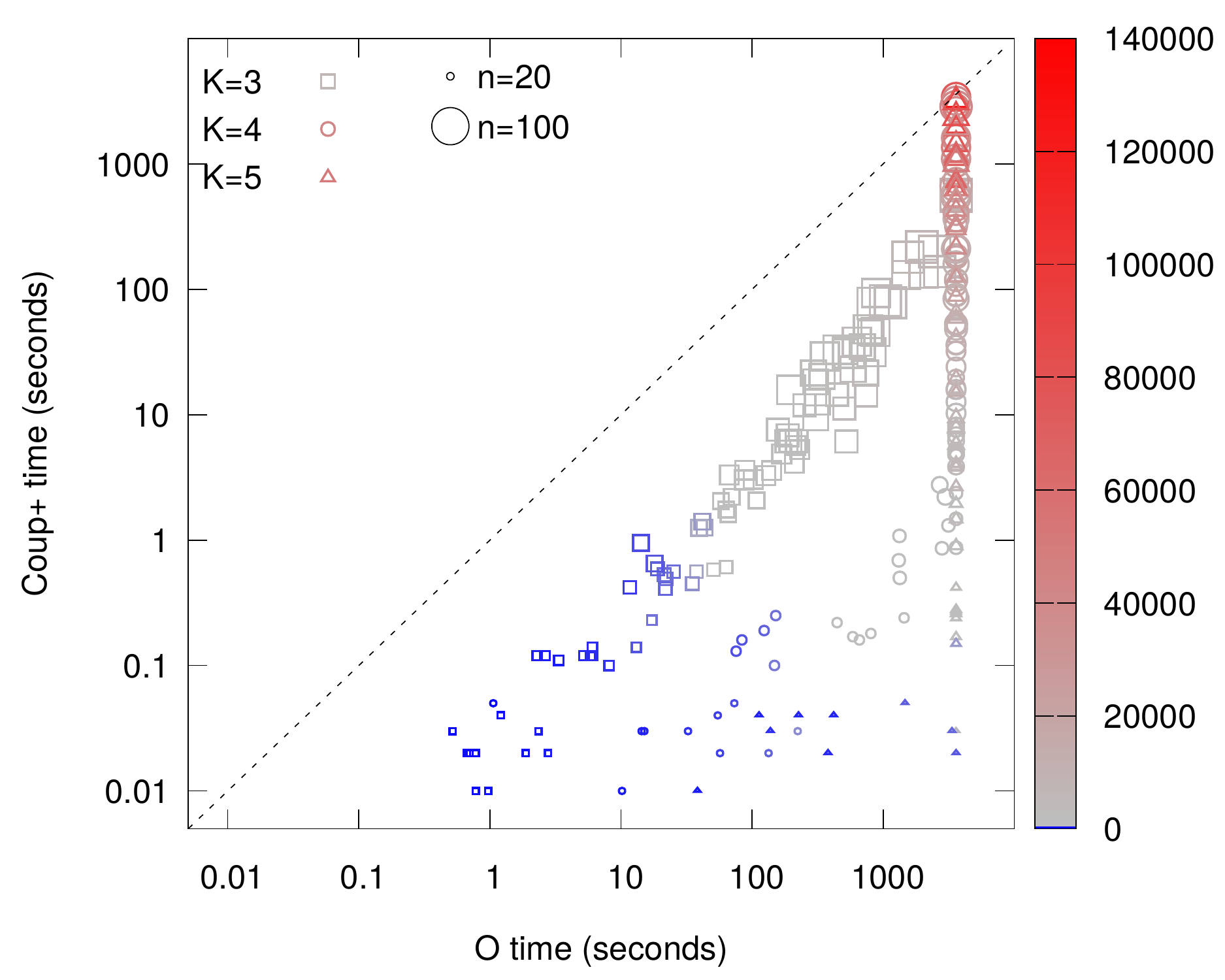}
		\caption{Knapsack scatter plot}
		\label{fig:knapSP}
	\end{minipage}
\end{figure}

The cumulative distribution plot for the knapsack instances is presented in Figure~\ref{fig:knapCD}. The results show a clear dominance of all network model algorithms over \aira{} and \kirlik{}. Overall, \coupp{} 
delivered the best results, solving 370 instances, whereas the configurations~\bu{} and~\bup{} were relatively weaker and solved 280 and 279 and instances, respectively. \aira{} and \kirlik{} solved 154 and 149, respectively. The figure also shows that the network model algorithms are considerably faster, as they solve more instances within seconds than~\aira{} and~\kirlik{} in one hour.


Figure~\ref{fig:knapSP} shows a scatter plot comparing \coupp{} with~\aira{} for instances with up to five objective functions. We removed from the plot instances with $\nobj = 6$ and $\nobj = 7$, since~\aira{} and \kirlik{} have considerably worse performance and the results do not provide much insight. \coupp{} was at least as efficient as~\aira{} and \kirlik{} in every knapsack instance tested, and only three instances that could not be solved by this network model configuration were solved by the others (one by~\td{} 
and two by \coup{}). The sizes of the Pareto frontier do not seem to affect the relative performance between~\coupp{} and~\aira{}; namely, \coupp{} and the other network model algorithms perform better than the previous state of the art in all cases, perhaps even more so in instances with smaller solution sets.


With respect to Pareto frontier compilation, the bidirectional strategy had the best results while bottom-up was relatively poor, independently from the inclusion of the label filtering VPO. 
Filtering affected network model algorithms in different ways, depending on the Pareto frontier compilation strategy used.  
The solution time differences between~\bu{} and~\bup{} are statistically significant ($p$-value of~$10^{-9}$), although in practical terms they perform similarly; the number of solved instances is almost the same (280 vs 279) and the average running time goes from 252 to 244 with filtering, a gain of 3\%  on average.  
The inclusion of filtering decreased the quality of the top-down algorithm, with 310 instances solved (as opposed to 323) and almost 40 seconds of additional computational time, on average, to solve instances solved by both algorithms (from 184 to 221, with $p$-value of $10^{-13}$). Finally, \coupp{} is significantly better than~\coup{} ($p$-value of $10^{-40}$); more instances were solved (370 in comparison with 362) in less time (average running time reduced from 318 to 228) and less variability (standard deviation reduced from 710 to 525).  

\subsection{Multiobjective Set Covering and Set Partitioning Problems} \label{sec:setCoveringandPartitioning}

We consider the multiobjective variants of the classical set covering and set
partitioning problems. Namely, let $A \in \mathbb{B}^{m \times n}$ be a binary
constraint coefficient matrix, and let $c^1,\dots,c^\nobj$ be $\nobj$ cost
vectors in $\mathbb{R}^n$. The multiobjective set covering problem (MSCP) is
defined as $$\min \left \{ ((c^1)^\top x, (c^2)^\top x, \dots, (c^\nobj)^\top
x) \,:\, Ax \ge \boldsymbol{1}, x \in \mathbb{B}^n \right \}.$$
The multiobjective set packing problem (MSPP) replaces ``min'' by ``max'' and
the inequality sign from ``$\ge$'' to ``$\le$'' in the definition above.



\paragraph{Network model construction.} The original networks are produced by
employing the DD transformation discussed in
Section~\ref{sec:bddTransformation}. In particular, we used the set covering DD
by \cite{BerHoeHoo11} and the set packing DD by \cite{BerCirHoeYun14}. In our
experiments, label filtering for both applications did not impact performance,
so we omit the corresponding results.

\paragraph{Computational evaluation.}

We experimented on 150 random instances with $n \in \{100, 150, 200\}$. Detailed results for the MCSP and the MSPP are presented in
Tables~\ref{tab:SetCoveringResults} and~\ref{tab:SetPackingResults},
respectively.

For the MSCP, \coup{} delivered the best results among the network model algorithms, solving 90 instances  with an average running time of 81 seconds (and a standard deviation of 188). \td{}  solved more instances (91), but its average runtime was higher (122 with a standard deviation of 263). Moreover, for the 88 instances solved by both, \coup{} had an average running time of 64 (with a standard deviation of 127), against 104 (standard deviation of 229) for~\td{}.

\begin{figure}[h]
	\centering
	\begin{minipage}[b]{0.54\textwidth}
	\includegraphics[width=0.89\textwidth]{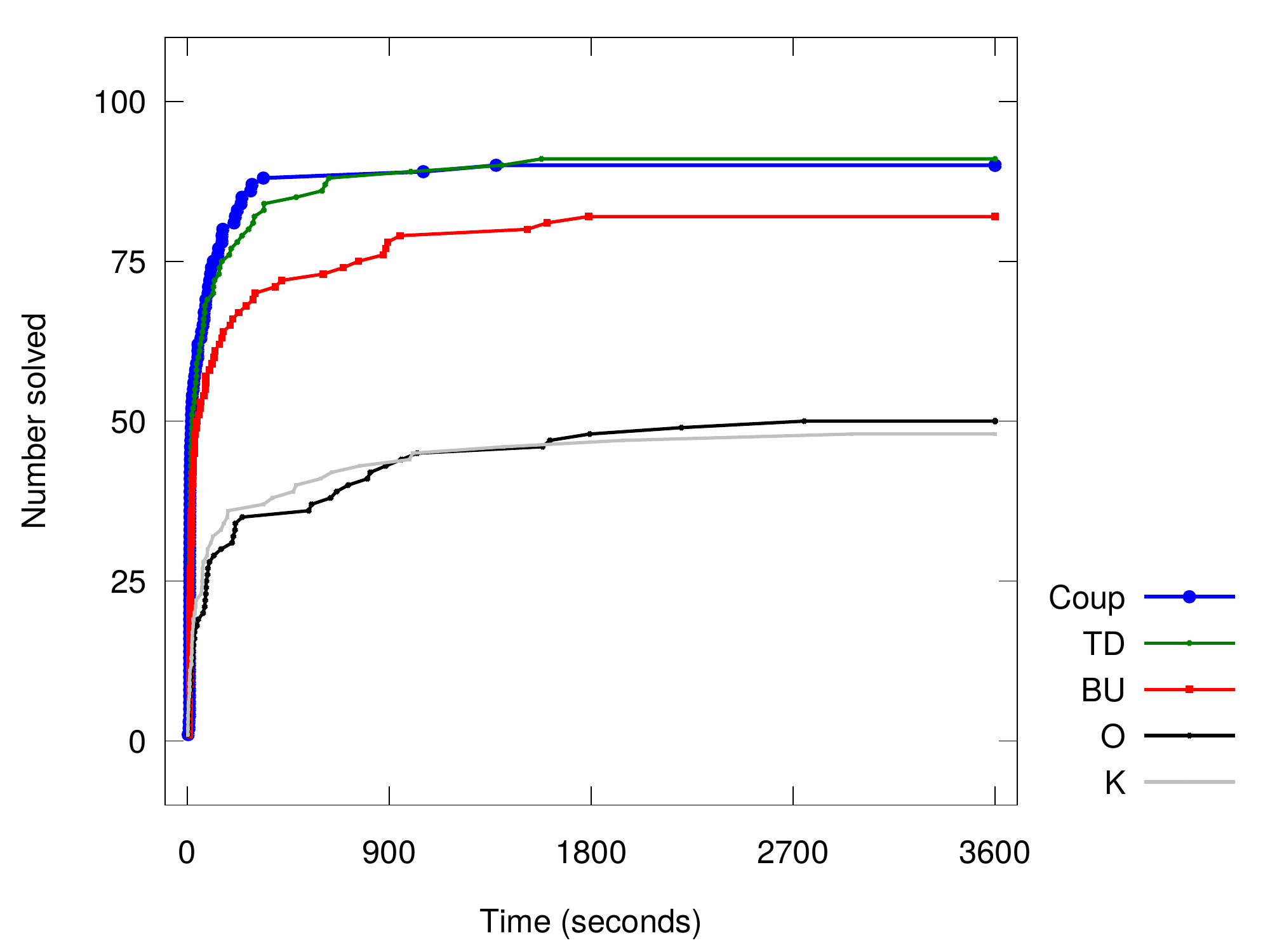}
	\caption{Set covering cumulative distribution plot}
	\label{fig:SCCD}
	\end{minipage}
	\hfill
	\begin{minipage}[b]{0.45\textwidth}
	\includegraphics[width=1\textwidth]{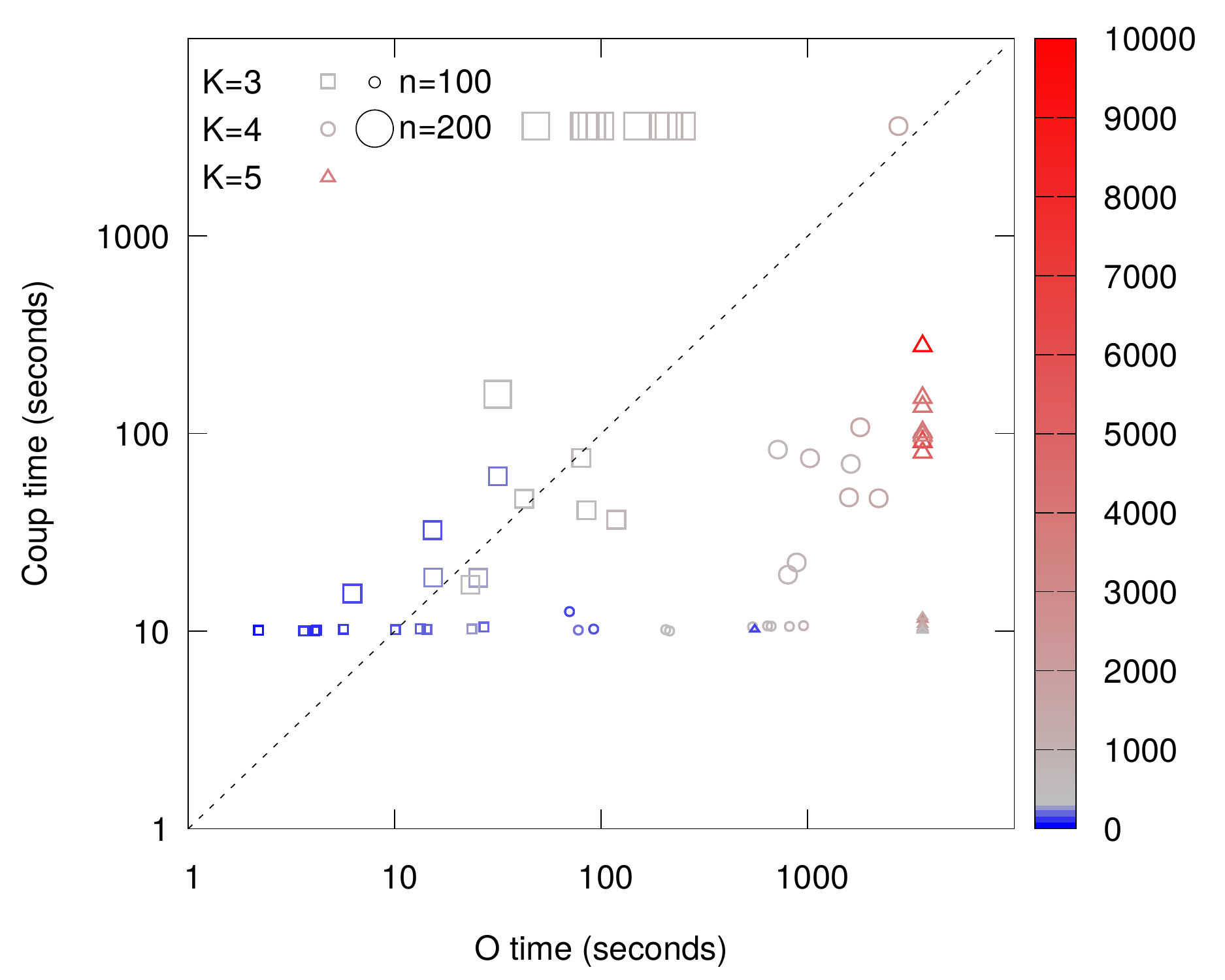}
\caption{Set covering scatter plot}
\label{fig:SCSP}
	\end{minipage}
\end{figure}

Algorithms \kirlik{} and \aira{} solved 48 and 50 instances, respectively. Among these, there were instances that the network model configurations could not enumerate the Pareto frontier (between 9 and 11, depending on the configuration employed). These instances were relatively large, typically of size $n = 200$, resulting in large network models that could not be solved within the given limits.

 Figure~\ref{fig:SCSP} depicts the relative performance of~\aira{} and~\coup{} on the MSCP, in particular elucidating the one configuration for which~\aira{} significantly outperformed~\coup{} (200 variables, $\nobj=3$). However, the same plot also suggests that~\coup{} is far more efficient for instances with relatively large Pareto frontiers. This suggests that the performance of the network models are less sensitive to the size of the Pareto frontier than~\kirlik{} and~\aira{}.
 

The results for the MSPP are presented in the cumulative distribution plot in Figure~\ref{fig:SPCD} and in the scatter plot in Figure~\ref{fig:SPSP}. 
\begin{figure}[h]
	\centering
	\begin{minipage}[b]{0.54\textwidth}
		\includegraphics[width=0.89\textwidth]{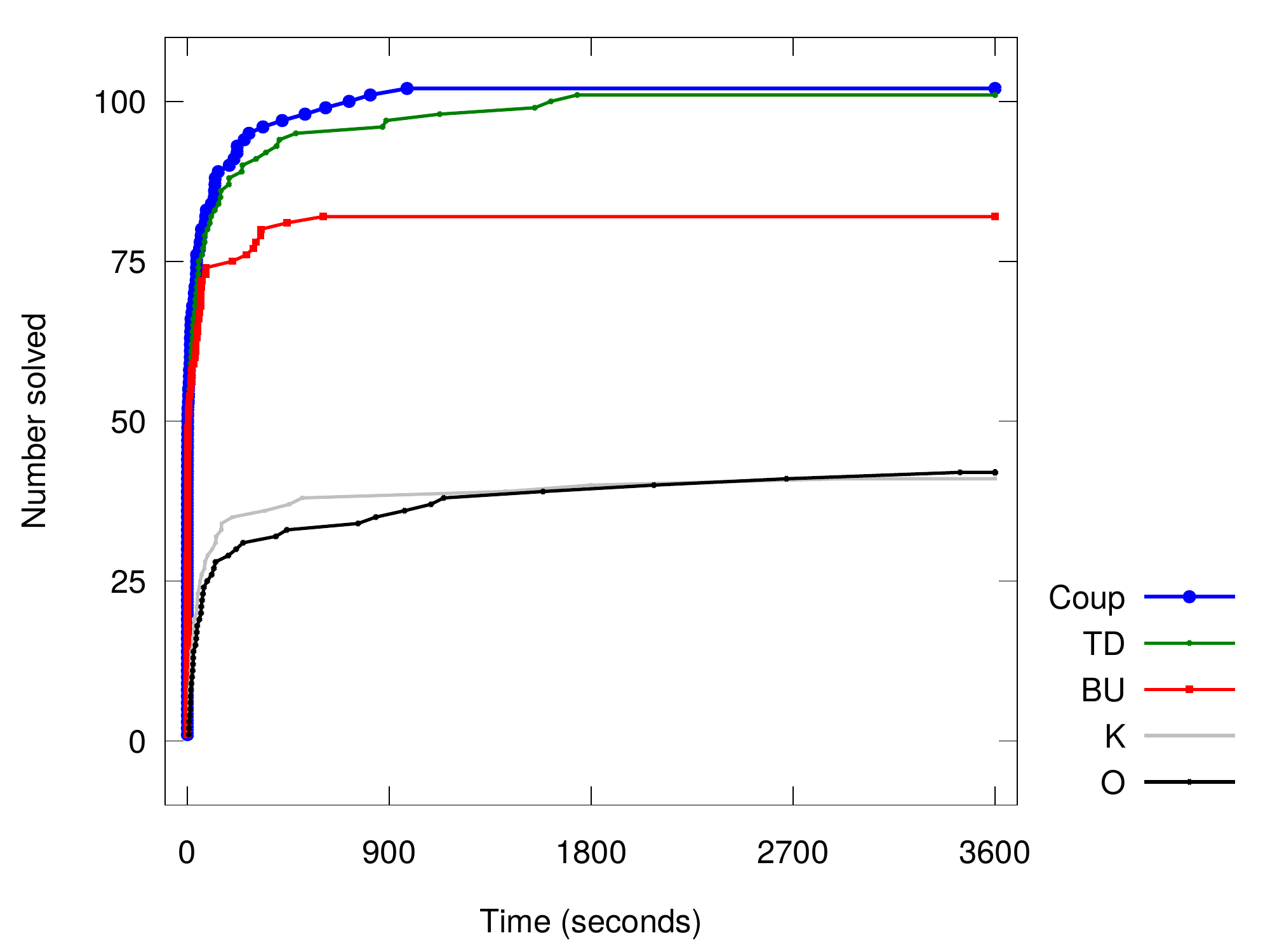}
		\caption{Set packing cumulative distribution plot}
		\label{fig:SPCD}
	\end{minipage}
	\hfill
	\begin{minipage}[b]{0.45\textwidth}
		\includegraphics[width=1\textwidth]{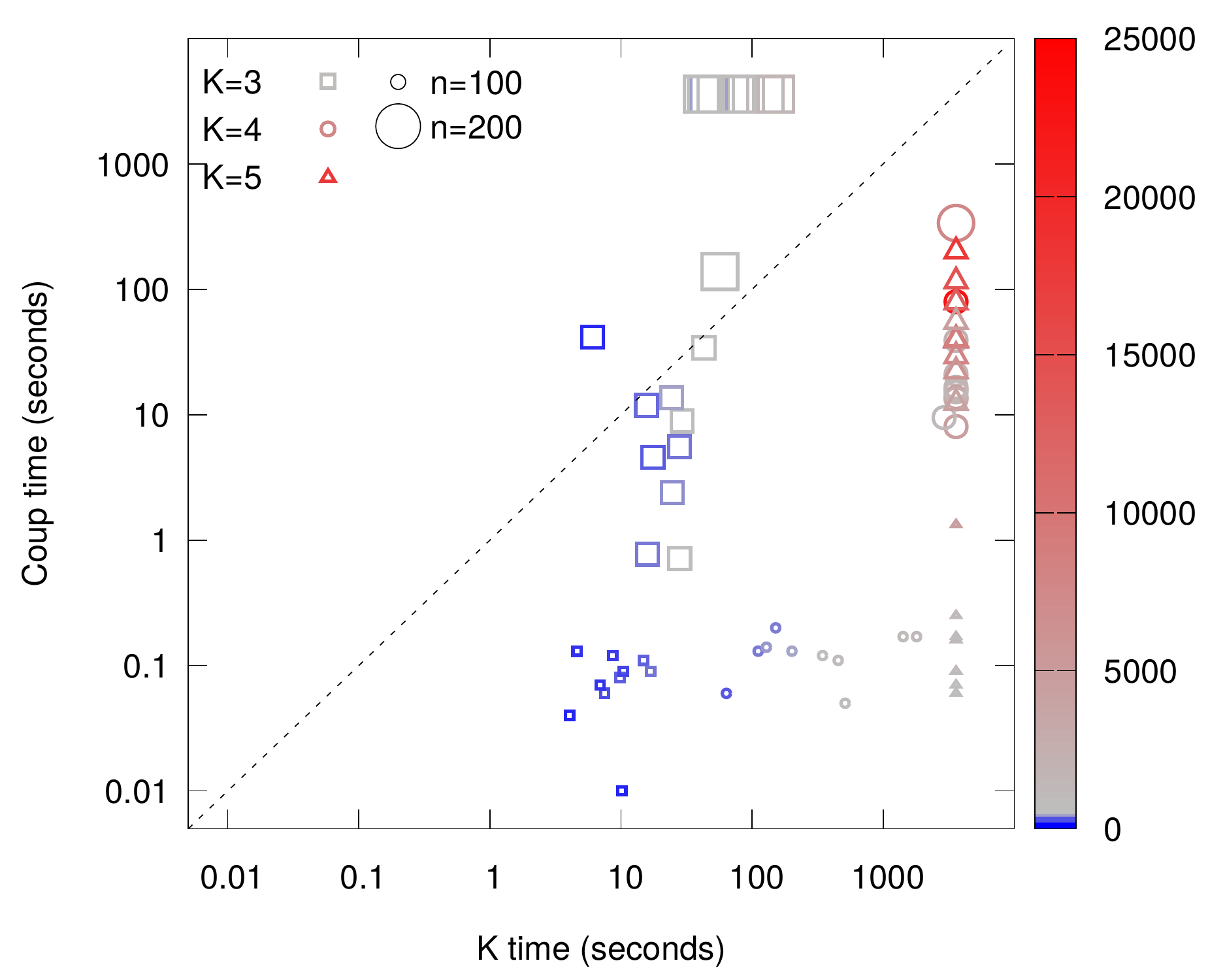}
		\caption{Set packing scatter plot}
		\label{fig:SPSP}
	\end{minipage}
\end{figure}
\coup{} also delivered the best results for this problem class, slightly outperforming \td{}; \coup{} solved one more instances than \td{} (102 vs 101) and had smaller running times (averages of 71 versus 123 and standard deviation of 170 versus 322). \coup{} and \td{} solved more instances of the MSPP than the MSCP, whereas \bu{}, \kirlik{}, and \aira{} had the inverse behavior. 
Algorithms \kirlik{} and~\aira{} again solved fewer instances (41 and 42) and, among these, between 9 and 10 instances (depending on the configuration) were not solved by the network models. The size of the instances played a role in the efficiency of the algorithms. Figure~\ref{fig:SPSP} shows that~\kirlik{} outperforms~\coup{} in some instances, but in relatively fewer cases than in the MSCP. Moreover, the Pareto frontier sizes have the same impact on the relative performance of the algorithms, as in the case of the MSCP, showing the robustness of network models.

\subsection{Multiobjective Traveling Salesperson Problem}

The multiobjective traveling salesperson problem (MTSP) is a generalization of the classical 
TSP where arcs are associated with multiple (often conflicting) distance measures. That is, given a graph $G=(V,E)$ 
with vertex set $V = \{1,\dots,n\}$ and where each 
edge $e \in E$ is associated with 
costs $c^1_{e}, \dots, c^{\nobj}_{e}$, MTSP 
asks for the nondominated Hamiltonian tours in $G$ with respect to edge costs. 

\paragraph{Network model construction.} The initial network is constructed
using the classical dynamic programming model for the TSP
\citep{Bertsekas2017}. Each state $s := (\bar{V}, v)$ is composed by a set $\bar{V}
\subseteq V$, representing the vertices that are still left to be visited, and
a vertex $v \in V \setminus \bar{V}$, representing the last visited vertex. The initial state is $s_0 := (V \setminus \{1\}, 1)$ (assuming
we start and end at vertex 1). The variable $x_j$ denotes the vertex that is
visited at the $j$-th position of the Hamiltonian tour; thus,
$\feasiblevaluestate_j((\bar{V}, v)) = x_j$. The transition function updates the set
of visited vertices, $\transition_j((\bar{V},v),x_j) = (\bar{V} \setminus \{x_j\}, x_j)$,
and the reward function is the negative of the distance travelled (since we are 
maximizing), $\reward_j((\bar{V},v), x_j) = (-c^1_{v,x_j}, \cdots, -c^{\nobj}_{v,x_j})$,
for $j = 1,\dots, n$. Finally, we establish a special terminal state with 
$\reward_{n+1}((\emptyset,v), x_j) = (-c^1_{v,1}, \cdots, -c^{\nobj}_{v,1})$
that represents the return to vertex 1.

\paragraph{Computational evaluation.} We experimented on 150 instances, with 10 instances for each configuration of $n \in \{5, 10, 15\}$ 
and $\nobj \in \{3, 4, 5, 6, 7\}$. We only depict 
results by \coup{}, which dominated all other network-based configurations, and \kirlik{}, which
also was superior to \aira{} in all scenarios tested. In particular, \kirlik{} uses the 
Miller-Tucker-Zemlin formulation of TSP \citep{miller1960integer} as in \cite{Ozlen2013}. Table \ref{tab:MTSPResults} depicts the results, 
where column $\mathsf{P}$ gives the average size of the Pareto frontiers (taking into account only closed instances), $\mathsf{S}$ gives the number of problems solved for the associated technique, and $\overline{t}$ provides the average time 
(out of 10 instances with the configuration); small running times were rounded up to 1 second.

{
\renewcommand{\arraystretch}{1.1}
\begin{table}[h]	
	\centering
	\small
	\caption{Multiobjective Traveling Salesperson Problem Results}
	\label{tab:MTSPResults}
	\begin{tabular}{|crr|rr|rr|}
		\multicolumn{3}{c}{} & \multicolumn{2}{c}{\coup{}} & \multicolumn{2}{c}{\kirlik{}} \\
		\hline \Tstrut \Bstrut
		$n$	 & $\nobj$ & $\mathsf{P}$ & $\mathsf{S}$ & $\overline{t}$ & $\mathsf{S}$ & $\overline{t}$ \\ 
		\hline \Tstrut \Bstrut
		5 & 3 & 6.9 & 10 & 1.0 & 10 & 1.0 \\
		& 4 & 8.7 & 10 & 1.0 & 10 & 1.0 \\
		& 5 & 8.1 & 10 & 1.0 & 10 & 1.2 \\
		& 6 & 11.0 & 10 & 1.0 & 10 & 2.8 \\
		& 7 & 10.9 & 10 & 1.0 & 5 & 11.4 \\
		\hline
		10 & 3 & 163.1 & 10 & 1.0 & 10 & 57.1 \\
		& 4 & 675.7 & 10 & 1.0 & 7 & 2021.4 \\
		& 5 & 2040.2 & 10 & 1.0 & 0 & - \\
		& 6 & 20080.5 & 10 & 1.0 & 0 & - \\
		& 7 & 9716.7 & 10 & 1.9 & 0 & - \\
		\hline
		15 & 3 & 670.7 & 10 & 3.2 & 7 & 2338.9 \\
		& 4 & 8328.5 & 10 & 41.2 & 0 & - \\
		& 5 & 55875.0 & 10 & 543.9 & 0 & - \\
		& 6 & 190447.3 & 8 & 2462.1 & 0 & - \\
		& 7 & - & 0 & - & 0 & - \\
		\hline 
	\end{tabular}
\end{table}
}

Our results show a complete dominance of~\coup{} over~\kirlik{}. Namely, \coup{} was superior to~\kirlik{} by at least one order of magnitude in all instances; for some configurations, \kirlik{} could not solve a single instance, whereas~\coup{} closed all scenarios within seconds (see e.g., $n = 10$ and $\nobj = 7$). Instances with 15 cities are very challenging for current state-of-the-art techniques; note that neither~\kirlik{} nor~\coup{} managed to solve any instance with $n = 15$ and~$\nobj = 7$; observe also that the size of the Pareto frontiers increase significantly with~$n$ and~$\nobj$. Nevertheless, \coup{} shows significant superiority in these scenarios as well; whereas \kirlik{} could not close any instances where $\nobj \geq 4$, \coup{} solved all instances with up to 5 objective functions in less than 10 minutes in average, and 8 out of 10 instances with $\nobj = 6$.

\subsection{Multiobjective Cardinality-Constrained Absolute Value Problem} 
\label{sec:nonlinear}

The multiobjective cardinality-constrained absolute value problem (MCCAVP)
is defined as
$$\min \left\{ \left( \left|(a^1)^\top x - b_1 \right|, \left|(a^2)^\top x - b_2 \right|, \hdots, \left|(a^K)^\top x - b_K \right| \right) \,:\, \boldsymbol{1}^\top x \leq C, x \in \mathbb{B}^n \right\},$$ 
where $a^1, \ldots, a^K \in \mathbb{Z}^n, b \in \mathbb{Z}^\nobj$, and $C \in \mathbb{Z}_+$.  
The MCCAVP is a multiobjective variant of the discrete $L_1$-norm minimization problems, classically applied in statistical data fitting and circuit optimization \citep{jong2012smoothing,Narula1982}. For instance, in data fitting problems each linear function represents a residual, and in the multiobjective case we wish to evaluate the Pareto frontier of nondominated fitting configurations according to each residual.

The MCCAVP illustrates how the procedure generalizes to multiobjective
nonlinear applications. If any of the $\nobj$ objective
functions is instead written as a linear function raised to the power of
$\alpha \in \mathbb{Z}_{\ge 1}$, the outer function can be replaced by the absolute value (if $\alpha$
is even) or simply ignored (if $\alpha$ is odd) without affecting the Pareto
frontier. The MCCAVP therefore provides a modeling framework for a wide-range
of objective functions.

\paragraph{Network model construction.} 

The initial network is constructed via a multiobjective recursive formulation as presented in Section~\ref{sec:recursiveModeling}. The recursive formulation is obtained by using a $(\nobj+1)$-dimensional
state variable $s:=(\theta, \gamma) \in \mathbb{R}^{\nobj} \times \mathbb{R}$, where $\theta_1, \dots, \theta_\nobj$ represent the partial evaluation of each $(a^k)^\top x$ for all $k \in [\nobj]$, and $\gamma$ is the number of variables that are set to one at that stage. The root state 
is $s_0 = (b,0)$. We cannot set a variable $x_j = 1$ if it exceeds the available capacity, i.e., $\feasiblevaluestate_j(s) := \{ v \in \mathbb{B} \,:\, \gamma + v \le C \}$. The transition functions update the partial evaluation of $(a^k)^\top x$ and the number of variables that are set to one; i.e.,
\[
\transition_j(s,v) = (\theta_1 + a^1_j \cdot v, \theta_2 + a^2_j \cdot v, \dots, \theta_\nobj + a^\nobj_j \cdot v, \gamma + v).
\]
for all $j \in [n]$. The reward function is the change in each objective function when transitioning from state $s=(\theta, \gamma)$ to another, that is, the $k$-th component of $\reward_j(s,v)$ is given by
\[
\bigg(  \reward_j(s,v) \bigg )_k := |\theta_k + a^k_j \cdot v - b_k| - |\theta_k - b_k|,
\]
for $k \in [k]$. To verify its validity, fix $k \in [\nobj]$ and consider any feasible $x =(v_1,\dots,v_n)$ and the associated state transitions $s_0 = (\theta^0, \gamma^0), s_1 = (\theta^1, \gamma^1), \dots, s_n =(\theta^n, \gamma^n)$. Observe that, by definition,
\begin{align*}
\sum_{j=1}^n \bigg( \reward_j(s_{j-1},v_j) \bigg)_k 
&= |\underbrace{\theta^0_k + a^k_1 \cdot v_1}_{\theta^1_k} - b_k| - |\theta^0_k - b_k| + |\underbrace{\theta^1_k + a^k_2 \cdot v_2}_{\theta^2_k} - b_k| - |\theta^1_k - b_k|   \\
& \quad + |\underbrace{\theta^2_k + a^k_3 \cdot v_3}_{\theta^3_k} - b_k| - |\theta^2_k - b_k| + \cdots + |\underbrace{\theta^{n-1}_k + a^k_n \cdot v_n}_{\theta^n_k} - b_k| - |\theta^{n-1}_k - b_k|  \\
&= |\theta^{n}_k - b_k| - \underbrace{|\theta^0_k - b_k|}_{0} 
=
\left| \sum_{j=1}^n a^k_j v_j - b_k \right |,
\end{align*}
which is the original objective for the $k$-th function.

\paragraph{Mixed-Integer Linear Programming Reformulation.} 
Since the original formulation for the MCCAVP is nonlinear, it cannot be directly input to~\kirlik{} and~\aira{}. We consider the following linear reformulation of the MCCAVP for~\kirlik{} and~\aira{}:
\begin{alignat*}{2}
\min_{x \in \mathbb{B}^n, \boldsymbol{1}^\top x \leq C} \ & \{ y_1,\hdots,y_\nobj\}  \\
\text{s.t.} \ & y_k \geq (a^k)^\top x - b_k, && k \in [\nobj], \\
				& y_k \geq -(a^k)^\top x + b_k, \qquad && k \in [\nobj].				
\end{alignat*}

\paragraph{Computational evaluation.} We experimented on 6,250 random instances. The detailed results for the MCCAVP are
presented in Table \ref{tab:AbsValueResults}. In particular, all
cases were solved by each network-based configuration. For this
section, we restrict the analysis of the results to 450 of these instances,
which have $M = 250$, $n \geq 15$, and $C \geq 30$, as the others were solved
within a few seconds.

The cumulative distribution plot is presented in Figure~\ref{fig:AVCD}. Whereas \kirlik{} and \aira{} solved less than 200 instances each, the network algorithms enumerated the Pareto frontier in less than 10 minutes in each case. Over the complete benchmark set, \bu{} delivered the best results. For the restricted set of 450  instances, \bu{} was also the best, although its results (average running time of 18 seconds with standard deviation of 41) were not different from those obtained by~\td{}  (average of 24 and standard deviation of 69) in a statistically significant way ($p$-value of $0.026$). 

Algorithm~\aira{} solved more instances than~\kirlik{} in the extended dataset (4,118 vs. 4,108), albeit with a higher runtime (on average 30\% larger). Alternatively, for the restricted family of instances,  \kirlik{} solved more instances (178 vs. 172) with a much shorter runtime (average of 258 and standard deviation of 492 vs. average of 432 and standard deviation of 801), thus suggesting that \kirlik{} outperforms \aira{} for harder instances. We therefore select~\kirlik{} for further comparison.

\begin{figure}[h]
	\centering
	\begin{minipage}[b]{0.58\textwidth}
		\includegraphics[width=0.78\textwidth]{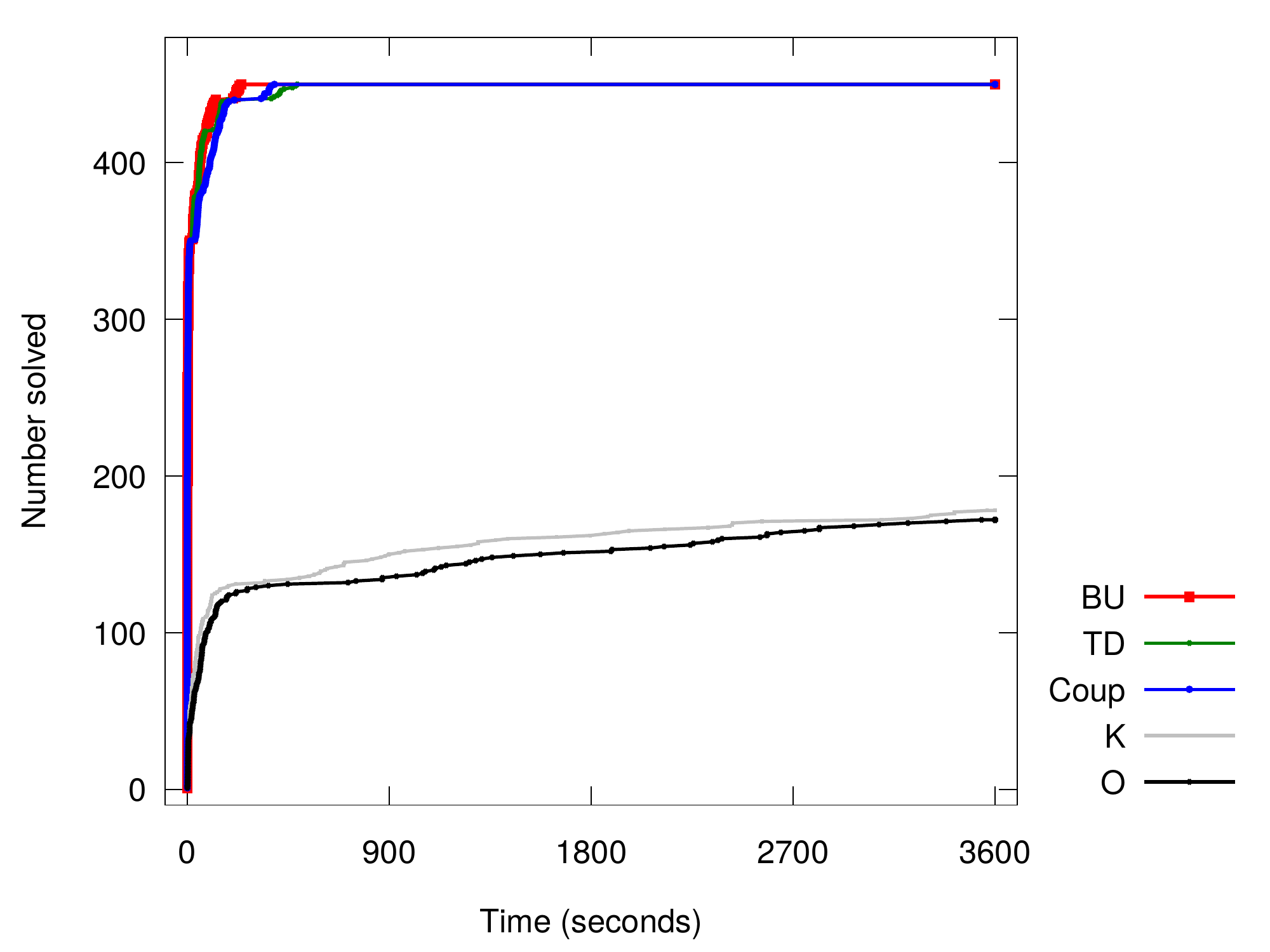}
		\caption{Absolute value cumulative distribution plot}
		\label{fig:AVCD}
	\end{minipage}
	\hfill
	\begin{minipage}[b]{0.41\textwidth}
		\includegraphics[width=1.02\textwidth]{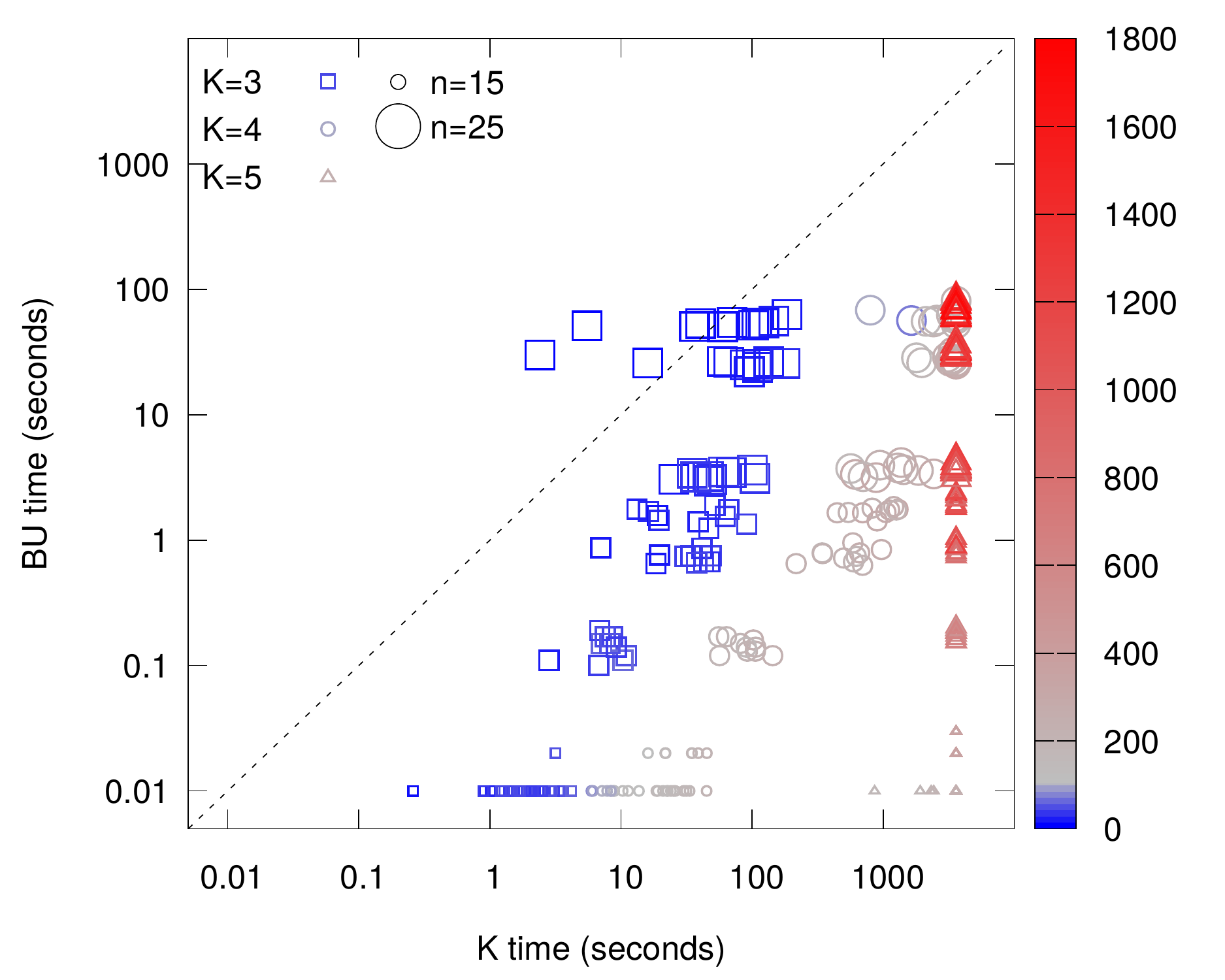}
		\caption{Absolute value scatter plot}
		\label{fig:AVSP}
	\end{minipage}
\end{figure}

Figure~\ref{fig:AVSP} shows a scatter plot comparing the performance of~\bu{} with~\kirlik{}. Algorithm~\kirlik{} outperforms~\bu{} in 5 instances by one order of magnitude (2 vs 30 seconds was the largest relative difference). All instances solved by \kirlik{} were also successfully addressed by \bu{}, with an average running time of 8 seconds and standard deviation of 17. Furthermore, similar to what was observed for the MSCP and MSPP, there is a positive correlation between the size of the Pareto frontiers and the relative superiority of network model procedures over~\kirlik{} and~\aira{}.

\section{Conclusion}
\label{sec:conclusionAndFutureWork}

This paper presents a novel framework for solving multiobjective discrete
optimization problems (MODOs) through a reformulation into \textit{network
models}, enhanced by validity-preserving operations that reduce the size of the
network while preserving the Pareto frontier. The generality of the framework
is established through application to five distinct problem classes, including
a nonlinear multiobjective optimization problem. The experimental evaluation
suggests that the proposed algorithm outperforms exisiting state-of-the-art
general MODO solvers in several multiobjective variants of classical operations
research problems. 

Our methodology assumes memory availability exceeds memory requirements for
constructing and storing network models. Since multiobjective optimization
problems become rapidly more challenging as problem size grows, this was only a
limiting factor for the largest of instances generated, most of which were
beyond the scope of other algorithms we tested against. As memory availability
in modern-day CPUs continues to grow, investigating algorithms specifically
designed to exploit this resource is of great interest, and this
paper provides an aimed step in this direction.





\bibliographystyle{plainnat} 
\bibliography{moobib,MOPalgos}

\newpage

\begin{appendix}


\section{Proofs}
\label{sec:proofs}

In this section we present the proofs for the structural results presented in the main text.

\subsection{Proof of Theorem \ref{thm:reduce}}
\label{subsec:ProofThmOne}
	Let~$\networkmodel'$ be the network model that results from the application of the three operations stated in Theorem \ref{thm:reduce} on nodes~$u_1, u_2$ of~$\networkmodel$. We claim that 
	there exists a one-to-one mapping between path-weights in~$\Pathset_{\networkmodel}$ and path-weights in~$\Pathset_{\networkmodel'}$; i.e., for each path in~$\Pathset_{\networkmodel}$, there is a path in~$\Pathset_{\networkmodel'}$ with same path-weight and vice-versa.
	
	First, every path $p \in \Pathset_{\networkmodel}$ that does not cross $u_2$ remains unchanged in $\Pathset_{\networkmodel'}$, thus the equivalence follows in this case. 
	Let us assume now that~$p = (a_{1}, \ldots,  a_{i},a_{i+1},\ldots,a_{n})$ is a path in  $\Pathset_{\networkmodel}$ such that $\arcterminal{a_i} = \arcroot{a_{i+1}} =  u_2$.
	 After the sequence of operations specified in the statement of the theorem, arc~$a_i$ is substituted for an arc~$a_i'$ in~$\networkmodel'$ such that $\arcroot{a_i'} = \arcroot{a_i}$, $\arcterminal{a_i'} = u_1$, and $\arcweightfunction(a_i') = \arcweightfunction(a_i)$, whereas arc~$a_{i+1}$ is substituted for an arc~$a_{i+1}'$ in~$\networkmodel'$ such that $\arcroot{a_{i+1}'} = u_1$, $\arcterminal{a_{i+1}'} = \arcterminal{a_{i+1}}$, and $\arcweightfunction(a_{i+1}') = \arcweightfunction(a_{i+1})$. These operations yield a one-to-one relationship between~$p$ and the path $p' = (a_{1}, \ldots,  a'_{i},a'_{i+1},\ldots,a_{n})$ in~$\networkmodel'$ which, by construction, is such that~$w(p') = w(p)$. Finally, as no additional paths are introduced,	
	 each path in~$\Pathset_{\networkmodel'}$ is also a path in~$\Pathset_{\networkmodel}$, 
	 so the result follows.	
	\qed

\subsection{Proof of Proposition \ref{prop:VPOcomplexity}} 
\label{subsec:ProofPropTwo}

Weight-shift and node-merge operations  can be implemented in a \emph{bottom-up} iterative procedure as follows. Starting from the penultimate layer, we construct the vector~$\tilde{c}(u)$ for each node~$u \in \layer{n}$ by taking the componentwise minimum arc-weight among arcs in $\arcs^{+}(u)$. Once~$\tilde{c}(u)$ has been obtained, the arc-weights can then be shifted up according to the VPO described in Proposition~\ref{prop:shift}. Each arc is inspected twice (for the identification of the minimum first and for the shift operation afterwards) and each node is visited once; thus, this operation can be performed over the complete network in time $\bigO( \nobj( |\nodes| +  |\arcs| ))$. 

Checking whether nodes~$u_1, u_2$ in layer~$\layer{i}$ satisfy the merging conditions of Theorem~\ref{thm:reduce} is an operation that can be performed in time $\bigO( \nobj  |\layer{i+1}| )$ if an adequate data structure is employed. Let~$\pi$ be an (arbitrary) ordering associated with the nodes of layer~$\layer{i+1}$ and let~$\pi_j$ denote the~$j$-th node of~$\layer{i+1}$ according to this ordering, with the first position starting from zero. Given~$\pi$, one can  construct a $\nobj |\layer{i+1}| $-dimensional  vector~$\xi(u)$ associated with each node~$u$ of layer~$\layer{i}$ such that $(\xi(u))_{\nobj j + k} = w( (u, \pi_j)  )_k $, i.e., entry $\nobj j + k$   contains the value of the~$k$-th coordinate of the 
arc-weight associated with the arc connecting $u$ to vertex~$\pi_{j}$ in~$\layer{i+1}$; some arbitrary value, such as $-\infty$, may be employed for entries associated with absent arcs. 
If the verification above is performed for each pair of nodes composing layer~$\layer{i}$, all merging operations can be verified in time $\bigO( \nobj |\layer{i}|^2 |\layer{i+1}| )$. As each pair of nodes in the network is compared at most once, we have a total running time of $\bigO( \nobj |\nodes|^3 )$. 

\qed

\subsection{Proof of Theorem \ref{thm:NPHardRemoval}}
\label{subsec:ProofThmTwo}
We consider a reduction from the following NP-Hard problem, written in terms of our notation and objective sense:

\smallskip

\noindent \textsc{Unconstrained Biobjective} \citep[Proposition 8.12]{ehrgott2006multicriteria}: Given $c^1, c^2 \in \mathbb{Z}^n$ and $d_1, d_2 \in \mathbb{Z}$, does there exist $x \in \mathbb{B}^n$ such that $\sum_{j=1}^n c^1_j x_j \ge d_1$ and $\sum_{j=1}^n c^2_j x_j \ge d_2$?

\smallskip

Let $c^1, c^2$, and $d := (d_1, d_2)$ define the parameters of an instance of the unconstrained biobjective problem. Equivalently, the problem asks if the given vector $d$ is dominated by a point in the image space of the MODO defined as
\begin{align*}
\tag{UCB} \label{eq:UCB} 
\max \left\{  (c^1)^\top x, (c^2)^\top x  : x \in \mathbb{B}^n \right  \}.
\end{align*}

We construct a network $\tilde{\networkmodel}$ where deciding if $d \in \nds(\tilde{\networkmodel})$ is equivalent to applying an arc-removal VPO for an appropriate subset of the arcs of $\tilde{\networkmodel}$. To this end, let us first consider a valid network model $\networkmodel = (\nodes, \arcs)$ for \ref{eq:UCB}. Since there are no constraints, we can construct a straightforward valid network model with one node per layer $\layer{j}$, $j \in [n+1]$, and two arcs $a_1, a_2$ connecting the node $u_1$ from layer 
$\layer{j}$ to node $u_2$ in layer $\layer{j+1}$, $j \in [n]$. In particular, $\arcweights{a_1} = \mathbf{0}$ (representing the assignment $x_j=0$) and $\arcweights{a_2} = \left( c^1_j, c^2_j \right)$ (representing the assignment $x_j=1$). The network $\networkmodel$ is depicted in Figure \ref{fig:UnconstraintMODO}\subref{fig:unconstNetwork}. Note that it has size polynomial in $n$, consisting of $n+1$ nodes and 
$2n$ arcs.

\begin{figure}[h]
\centering
\subfigure[$\networkmodel$]{
\begin{tikzpicture}[scale=0.25][font=\sffamily,\tiny]
	\node [draw,circle,minimum size=0.68cm] (root) {$\rootnode_{\networkmodel}$};
	\node [draw,circle,minimum size=0.68cm,below of = root,yshift = -4.5cm] (terminal) {$\terminalnode_{\networkmodel}$};
	
	\node [draw,circle,minimum size=0.68cm,below of = root,yshift = -0.55cm] (node1) {};
	\node [below of = node1,yshift = -0.25cm] (node2) {};
	\node [draw,circle,minimum size=0.68cm,below of = node2,yshift = -0.1cm] (node3) {};
	
	\path[very thick,dotted,shorten >= -0pt,every node/.style={font=\sffamily\small}] (root) edge[bend right] node [left] {$(c^1_1, c^1_2)$} (node1);
	\path[very thick,dotted,shorten >= -1.2pt,every node/.style={font=\sffamily\small}] (node1) edge[bend right] node [left] {} (node2);
	\path[very thick,dotted,shorten >= -0pt,every node/.style={font=\sffamily\small}] (node2) edge[bend right] node [left] {} (node3);
	\path[very thick,dotted,shorten >= -0pt,every node/.style={font=\sffamily\small}] (node3) edge[bend right] node [left] {$(c^1_n, c^2_n)$} (terminal);

	\path[very thick,dotted,shorten <= -0pt,every node/.style={font=\sffamily\small}] (root) edge[bend left] node [right] {$\textbf{0}$} (node1);
	\path[very thick,dotted,shorten <= -0.5pt,every node/.style={font=\sffamily\small}] (node1) edge[bend left] node [right] {} (node2);
	\path[very thick,dotted,shorten <= -0.1pt,every node/.style={font=\sffamily\small}] (node2) edge[bend left] node [right] {} (node3);
	\path[very thick,dotted,shorten <= -0pt,every node/.style={font=\sffamily\small}] (node3) edge[bend left] node [right] {$\textbf{0}$} (terminal);
	
\end{tikzpicture}
\label{fig:unconstNetwork}
}
\hspace*{2.5cm}
\subfigure[$\networkmodel'$] {
\begin{tikzpicture}[scale=0.3][font=\sffamily,\tiny]
	\node [draw,circle,minimum size=0.68cm] (rootnew) {$\rootnode_{\networkmodel'}$};
	\node [draw,circle,,minimum size=0.68cm,right of = rootnew,yshift = -1.25cm, xshift = 0.1cm] (node2) {};
	
	\node [draw,circle,minimum size=0.68cm,left of = rootnew,yshift = -1.25cm, xshift = -0.1cm] (root) {$\rootnode_{\networkmodel}$};
	\node [draw,circle,minimum size=0.68cm,below of = root,yshift = -1cm] (terminal) {$\terminalnode_{\networkmodel}$};
	
	\node [draw,circle,minimum size=0.68cm]  (node3) at (root -| node2) {};
	\node [draw,circle,minimum size=0.68cm] (node4) at (terminal -| node2) {};

	\node [draw,circle,minimum size=0.68cm,below of = rootnew,yshift = -3.5cm] (terminalnew) {$\terminalnode_{\networkmodel'}$};
	\path[->](rootnew) edge node [xshift=-0mm,yshift=2mm,left] {\footnotesize $\arcweightfunction(a') = \textbf{0}$} (root);
	\path[->](rootnew) edge node [xshift=-0.6mm,yshift=2mm,right] {\footnotesize $ d =\arcweightfunction(a'')$} (node2);
	\draw[dotted,very thick] (node3) edge node [xshift=-1mm,right] {\footnotesize $\textbf{0}$} (node4);
	\path[->](node4) edge node [xshift=-1mm,yshift=-1mm,right] {\footnotesize $\textbf{0}$} (terminalnew);
	\path[->](terminal) edge node [xshift=-0.2mm,yshift=-1mm,left] {\footnotesize $\textbf{0}$} (terminalnew);
	
	\node [below of = root,yshift = -0cm] (dummy) {};
	\path[very thick,dotted,shorten >= -1.2pt,every node/.style={font=\sffamily\small}] (root) edge[bend right] node [left] {} (dummy);
	\path[very thick,dotted,shorten <= -1.2pt,every node/.style={font=\sffamily\small}] (dummy) edge[bend right] node [left] {} (terminal);
	\path[very thick,dotted,shorten >= -1.2pt,every node/.style={font=\sffamily\small}] (root) edge[bend left] node [right] {} (dummy);
	\path[very thick,dotted,shorten <= -1.2pt,every node/.style={font=\sffamily\small}] (dummy) edge[bend left] node [right] {} (terminal);
\end{tikzpicture}
\label{fig:extendednetwork}
}
\caption{A network model $\networkmodel$ of a binary unconstrained MODO (a), and an extended network model $\tilde{\networkmodel}$ of $\networkmodel$ (b), used in the proof of Theorem \ref{thm:NPHardRemoval}.}
\label{fig:UnconstraintMODO}
\end{figure}

We now design a new network model~$\tilde{\networkmodel}$ by extending~$\networkmodel$ as depicted in Figure \ref{fig:UnconstraintMODO}\subref{fig:extendednetwork}. 
The root node~$\rootnode_{\tilde{\networkmodel}}$ is the arc-root of two arcs~$a'$ and~$a''$. 
Arc~$a'$ has arc-weight~$\arcweightfunction(a') = \textbf{0}$  and connects~$\rootnode_{\tilde{\networkmodel}}$ to the corresponding~$\rootnode_{\networkmodel}$ in the new network. 
Arc~$a''$ has arc-weight~$\arcweightfunction(a'') = d$, and~$\tilde{\networkmodel}$ contains a single path~$p'' = (\rootnode_{\tilde{\networkmodel}}, \arcterminal{a''},\ldots,\terminalnode_{\tilde{\networkmodel}})$ connecting $\rootnode_{\tilde{\networkmodel}}$  to~$\terminalnode_{\tilde{\networkmodel}}$ that traverses node~$\arcterminal{a''}$; 
the arc-weight of all arcs in~$p''$ (except~$a''$) is equal to zero.   Finally, a single arc with arc-weight $\textbf{0}$ connects~$\terminalnode_{\tilde{\networkmodel}}$ to the associated~$\terminalnode_{\networkmodel}$ in the new node. 

Let~$\arcs'$ be the subset of~$\arcs_{\tilde{\networkmodel}}$ containing all arcs that compose path~$p''$. If~$d$ is dominated by~$\nds(\tilde{\networkmodel})$, any arc~$a$ in $\arcs'$ may be removed from~$\tilde{\networkmodel}$ without changing its Pareto frontier, i.e., $\paretofrontier{\tilde{\networkmodel}} = \paretofrontier{\tilde{\networkmodel} - a}$. Otherwise, by construction, $d$ necessarily belongs to~$\paretofrontier{\tilde{\networkmodel}}$,  thus 
$d \in \paretofrontier{\networkmodel}$. Hence, no arc is removed from~$\arcs'$. Since all steps are polynomially bounded on $n$, the result follows.
\qed
	
\subsection{Proof of Theorem \ref{thm:subBDDarcRemoval}}
\label{subsec:ProofThmThree}
Suppose by contradiction that there exists a pair of nodes $u$ and $v$ that are isolating in $\networkmodel$, together with an arc~$a$ whose removal is a VPO for~$\networkmodel[u,v]$ but not for~$\networkmodel$.  By definition, if such an arc~$a$  exists, then there is a path $p$ in $\Pathset_{\networkmodel}$ from $\rootnode_\networkmodel$ to $\terminalnode_\networkmodel$ containing~$a$ such that $w(p) \in \paretofrontier{\networkmodel}$; additionally, $\Pathset_{\networkmodel}$ shall not contain any other path~$p''$ in $\Pathset_{\networkmodel}$ such that $w(p'') \strictlydominates  w(p)$.

Since $a \in \arcs_{\networkmodel[u,v]}$, and $u$ and $v$ are isolating, $p$ must traverse $u$ and $v$. Let $p'$ be the subpath of $p$ directed from $u$ to $v$. Additionally, let $p^1$ be the portion of $p$ from $\rootnode_\networkmodel$ to $u$, and let $p^2$ be the portion of $p$ from $v$ to $\terminalnode_\networkmodel$. As the removal of $a$ from $\networkmodel[u,v]$ does not alter $\paretofrontier{\networkmodel[u,v]}$, there must exist a path $p''$ from $u$ to $v$ for which $w(p') \strictlydominated w(p'')$. Let $\tilde{p}$ be the path in $\networkmodel$ constructed from concatenating $p^1, p''$, and $p^2$. Then,  $w(\tilde{p}) \strictlydominates w(p)$ (as $w(\tilde{p}) \neq w(p)$ by assumption), which
contradicts that $w(p) \in \paretofrontier{\networkmodel}$. 
\qed

\section{Recursive Model Example}
\label{ec:recursive}
\vspace*{-0.2cm}
\begin{align*}
\mathcal{X} = \{ x \in \mathbb{B}^7 \ : \ x_1 + x_2 + x_3 \leq 1, x_2 + x_3  + x_4 \leq 1, x_4 + x_5 \leq 1, x_4 + x_6 \leq 1,  x_5 + x_7 \leq 1,x_6 + x_7 \leq 1 \}.
\end{align*}

\vspace*{-0.3cm}
\begin{figure}[h]
	\centering
	\begin{tikzpicture}[scale=0.3][font=\sffamily,\tiny]
	
	\node [draw,rectangle,scale=1] (r) at (0,0) {$s_0 = (0,0,0,0,0,0)$};
	\node [draw,rectangle,scale=1] (a) at (-6,-5) {$u^2_1 = (0,0,0,0,0,0)$};
	\node [draw,rectangle,scale=1] (b) at (6,-5) {$u^2_2= (1,0,0,0,0,0)$};
	\node [draw,rectangle,scale=1] (c) at (-12,-11) {$u^3_1= (1,1,0,0,0,0)$};
	\node [draw,rectangle,scale=1] (d) at (0,-11) {$u^3_2= (0,0,0,0,0,0)$};
	\node [draw,rectangle,scale=1] (e) at (12,-11) {$u^3_3= (1,0,0,0,0,0)$};
	\node [draw,rectangle,scale=1] (f) at (-6,-17) {$u^4_1= (1,1,0,0,0,0)$};
	\node [draw,rectangle,scale=1] (g) at (6,-17) {$u^4_2= (1,0,0,0,0,0)$};
	\node [draw,rectangle,scale=1] (h) at (-6,-23) {$u^5_1= (1,1,0,0,0,0)$};
	\node [draw,rectangle,scale=1] (i) at (6,-23) {$u^5_2= (1,1,1,1,0,0)$};
	\node [draw,rectangle,scale=1] (j) at (-16,-29) {$u^6_1= (1,1,1,0,1,0)$};
	\node [draw,rectangle,scale=1] (k) at (1.8,-29) {$u^6_2= (1,1,1,0,0,0)$};
	\node [draw,rectangle,scale=1] (l) at (12,-29) {$u^6_3= (1,1,1,1,0,0)$};
	\node [draw,rectangle,scale=1] (m1) at (-22,-35) {$u^7_1= (1,1,1,1,1,1)$};
	\node [draw,rectangle,scale=1] (m2) at (-13,-35) {$u^7_2= (1,1,1,1,1,0)$};
	\node [draw,rectangle,scale=1] (m) at (-3,-35) {$u^7_3= (1,1,1,1,0,1)$};
	\node [draw,rectangle,scale=1] (n) at (8,-35) {$u^7_4= (1,1,1,1,0,0)$};
	\node [draw,rectangle,scale=1] (t) at (0,-42) {$s_\terminalnode = (1,1,1,1,1,1)$};
	
	
	\path[->,shorten >=0.08cm](r) edge node [xshift=-1mm,yshift=0.2mm,left] {$x_1 = 0$} (a);
	\path[->,shorten >=0.08cm](r) edge node [xshift=1mm,yshift=0.2mm,right] {$x_1 = 1$} (b);
	
	\path[->,shorten >=0.08cm](a) edge node [xshift=-2mm,yshift=-1mm,left] {$x_2 = 1$} (c);
	\path[->,shorten >=0.08cm](a) edge node [xshift=0.3mm,yshift=-1mm,left] {$x_2 = 0$} (d);
	\path[->,shorten >=0.08cm](b) edge node [xshift=2.3mm,yshift=-1mm,right] {$x_2 = 0$} (e);
	
	\path[->,shorten >=0.08cm](c) edge node [left,xshift=-0.8mm] {$x_3 = 0$} (f);
	\path[->,shorten >=0.08cm](d) edge node [left] {$x_3 = 1$} (f);
	\path[->,shorten >=0.08cm](d) edge node [right] {$x_3 = 0$} (g);
	\path[->,shorten >=0.08cm](e) edge node [right] {$x_3 = 0$} (g);
	
	\path[->,shorten >=0.08cm](f) edge node [left] {$x_4 = 0$} (h);
	\path[->,shorten >=0.08cm](g) edge node [left,xshift=-0.8mm] {$x_4 = 0$} (h);
	\path[->,shorten >=0.08cm](g) edge node [right] {$x_4 = 1$} (i);
	
	\path[->,shorten >=0.08cm](h) edge node [left,xshift=-0.8mm] {$x_5 = 1$} (j);
	\path[->,shorten >=0.08cm](h) edge node [right,xshift=0.8mm] {$x_5 = 0$} (k);
	\path[->,shorten >=0.08cm](i) edge node [right,xshift=0.2mm] {$x_5 = 0$} (l);
	
	\path[->,shorten >=0.08cm](j) edge node [left,xshift=-0.05cm] {$x_6 = 1$} (m1);
	\path[->,shorten >=0.08cm](j) edge node [right] {$x_6 = 0$} (m2);
	\path[->,shorten >=0.08cm](k) edge node [right] {$x_6 = 1$} (m);
	\path[->,shorten >=0.08cm](k) edge node [right] {$x_6 = 0$} (n);
	\path[->,shorten >=0.08cm](l) edge node [right] {$x_6 = 0$} (n);
	
	\path[->,shorten >=0.08cm](m1) edge node [left,xshift=-0.3cm] {$x_7 = 0$} ([xshift=-3cm] t.north);
	\path[->,shorten >=0.08cm](m2) edge node [left,xshift=-0.15cm] {$x_7 = 0$} (t);
	\path[->,shorten >=0.08cm](m) edge node [left,xshift=-0.05cm] {$x_7 = 0$} (t);
	\path[->,shorten >=0.08cm](n) edge node [left,xshift=-0.1cm] {$x_7 = 0$} (t);
	\path[->,shorten >=0.08cm](n.south) edge node [right] {$x_7 = 1$} ([xshift=3cm] t.north);
	\end{tikzpicture}
	\caption{A valid network model for $\modo$ in Example~\ref{ex:modo} obtained by recursive formulation given in Example \ref{ex:modorec}.}
	\label{fig:recursivemodelexample}
\end{figure}

\section{Data Generation}

\paragraph{Multiobjective 0-1 Knapsack Problem.} We experimented on 450 instances of the knapsack problem with
$\nobj \in \{ 3,4,\dots,7 \}$ and $n \in \{20, 30, \dots, 100\}$. Instances were generated as in \cite{kirlik2014new}: each profit $p^k_j$ and weight $w_j$ was randomly draw from 
the interval $[1,1000]$, for $j \in [n]$ and $k \in [\nobj]$. The capacity of the knapsack was set to $W := \lceil 0.5 \sum_{i=1}^n w_i \rceil$. Ten instances were generated for each
$(n, \nobj)$ pair.

\bigskip

\paragraph{Multiobjective Set Covering and Set Partitioning Problems.} We generated 150 random instances based on the previous work by \cite{Stidsen2014}. Specifically, we considered $n \in \{ 100,150,200 \}$ variables, $m=n/5$ constraints, and
fixed $10$ variables per constraint (i.e., for every $k = 1,\dots,m$, we pick 10 elements
of the $k$-th row of $A$ to be 1 uniformly at random). We considered $\nobj=3,\dots,7$ objectives. We sampled 10 instances per $(n,\nobj)$ pair. Each matrix was used both for the MSCP and the MSPP by selecting the appropriate inequality direction in the ensuing constraint system.

\bigskip

\paragraph{Multiobjective Traveling Salesperson Problem.} We experimented on 150 instances with $n \in \{5, 10, 15\}$ vertices and $\nobj \in \{3, 4, 5, 6, 7\}$. These 
 instances are generated as in \cite{ozpeynirci2010exact}: for each $\nobj$, we generated integer coordinates for $n$ cities on a $1000 \times 1000$ square (uniformly at random) and used Euclidean distances to create the distance matrix. 

\bigskip

\paragraph{Multiobjective Cardinality-Constrained Absolute Value Problem.} We experimented on 6,250 instances of the problem. Each was randomly generated with $K \in \{3,4,5,6,7\}$  and $n \in \{5,10,15,20,25\}$. For all $k \in [K]$, we drew the components of $a^k$ uniformly
at random from the set $[-M,M]$, where we considered $M \in \{ 50, 100, 150, 200, 250 \}$. We let $b_k = \lfloor \boldsymbol{1}^\top a^k /2 \rfloor$ for all $k \in [\nobj]$. For the cardinality constraints, we let $C = \lfloor n \delta  \rfloor$ for which we considered $\delta \in \{0.1, 0.2, 0.3, 0.4, 0.5\}$.

\section{Detailed Computational Results}

In this section we provide a comprehensive reporting of the computational results. Each line in each table presents the aggregated results of instances whose dimensions are described in the leftmost column: $n$ indicates the number of variables and~$\nobj$ gives the number of objective functions. The other columns are arranged in groups of three and present the results for the algorithm indicated on the top. Column $S$ gives the number of instances that were solved within the predefined memory and time limits, and~$\bar{t}$ is the average runtime of the algorithm to solve these instances. Column $(\texttt{M},\texttt{T})$ report pairs~$(m,t)$, indicating that the algorithm could not solve~$m$ and~$t$ instances of that particular configuration because it exceeded the predefined memory and time limits, respectively.

\begin{landscape}
	
	\begin{table}[]
		\centering
		\tiny
		\caption{Knapsack results}
		\label{tab:KnapsackResults}	
		\begin{tabular}{P{2ex}P{1.5ex}|P{1ex}P{6ex}P{6ex}|P{1ex}P{6ex}P{6ex}|P{1ex}P{6ex}P{6ex}|P{1ex}P{6ex}P{6ex}|P{1ex}P{6ex}P{6ex}|P{1ex}P{6ex}P{6ex}|P{1ex}P{6ex}P{6ex}|P{1ex}P{6ex}P{6ex}} 
			\multicolumn{2}{c}{}  & 
			\multicolumn{3}{c}{\texttt{K}} & 
			\multicolumn{3}{c}{\texttt{O}} & 
			\multicolumn{3}{c}{\texttt{TD}} & 
			\multicolumn{3}{c}{\texttt{BU}} & 
			\multicolumn{3}{c}{\texttt{Coup}} & 
			\multicolumn{3}{c}{\texttt{TD$^+$}} & 
			\multicolumn{3}{c}{\texttt{BU$^+$}} & 
			\multicolumn{3}{c}{\texttt{Coup$^+$}} 
			\\ \hline \Tstrut \Bstrut
			$n$	&	$\nobj$ 	& $\mathsf{S}$	& $\bar{t}$	& $(\mathsf{M},\mathsf{T})$ 	& $\mathsf{S}$	& $\bar{t}$	& $(\mathsf{M},\mathsf{T})$  	& $\mathsf{S}$	& $\bar{t}$	& $(\mathsf{M},\mathsf{T})$  	& $\mathsf{S}$	& $\bar{t}$	& $(\mathsf{M},\mathsf{T})$  	& $\mathsf{S}$	& $\bar{t}$	& $(\mathsf{M},\mathsf{T})$  	& $\mathsf{S}$	& $\bar{t}$	& $(\mathsf{M},\mathsf{T})$  	& $\mathsf{S}$	& $\bar{t}$	& $(\mathsf{M},\mathsf{T})$  	& $\mathsf{S}$	& $\bar{t}$	& $(\mathsf{M},\mathsf{T})$  \\ \hline \hline \Bstrut\Tstrut
			20  & 3 & 10 & 0.93    & (0,0)  & 10 & 1.26    & (0,0)  & 10 & 0.02    & (0,0)  & 10 & 0.02    & (0,0)  & 10 & 0.02    & (0,0)  & 10 & 0.02    & (0,0)  & 10 & 0.02    & (0,0)  & 10 & 0.02    & (0,0)  \\
			20  & 4 & 10 & 34.41   & (0,0)  & 10 & 61.45   & (0,0)  & 10 & 0.03    & (0,0)  & 10 & 0.03    & (0,0)  & 10 & 0.03    & (0,0)  & 10 & 0.02    & (0,0)  & 10 & 0.03    & (0,0)  & 10 & 0.03    & (0,0)  \\
			20  & 5 & 6  & 200.79  & (0,4)  & 8  & 766.91  & (0,2)  & 10 & 0.03    & (0,0)  & 10 & 0.03    & (0,0)  & 10 & 0.03    & (0,0)  & 10 & 0.03    & (0,0)  & 10 & 0.03    & (0,0)  & 10 & 0.03    & (0,0)  \\
			20  & 6 & 0  & -       & (2,8)  & 0  & -       & (0,10) & 10 & 0.03    & (0,0)  & 10 & 0.04    & (0,0)  & 10 & 0.03    & (0,0)  & 10 & 0.03    & (0,0)  & 10 & 0.03    & (0,0)  & 10 & 0.02    & (0,0)  \\
			20  & 7 & 0  & -       & (0,10) & 0  & -       & (0,10) & 10 & 0.04    & (0,0)  & 10 & 0.07    & (0,0)  & 10 & 0.04    & (0,0)  & 10 & 0.05    & (0,0)  & 10 & 0.05    & (0,0)  & 10 & 0.04    & (0,0)  \\ \hline \Tstrut \Bstrut
			30  & 3 & 10 & 5.22    & (0,0)  & 10 & 6.98    & (0,0)  & 10 & 0.14    & (0,0)  & 10 & 0.22    & (0,0)  & 10 & 0.16    & (0,0)  & 10 & 0.13    & (0,0)  & 10 & 0.21    & (0,0)  & 10 & 0.13    & (0,0)  \\
			30  & 4 & 10 & 676.25  & (0,0)  & 10 & 452.18  & (0,0)  & 10 & 0.18    & (0,0)  & 10 & 0.53    & (0,0)  & 10 & 0.2     & (0,0)  & 10 & 0.18    & (0,0)  & 10 & 0.54    & (0,0)  & 10 & 0.18    & (0,0)  \\
			30  & 5 & 0  & -       & (1,9)  & 0  & -       & (0,10) & 10 & 0.3     & (0,0)  & 10 & 1.42    & (0,0)  & 10 & 0.34    & (0,0)  & 10 & 0.38    & (0,0)  & 10 & 1.27    & (0,0)  & 10 & 0.26    & (0,0)  \\
			30  & 6 & 0  & -       & (0,10) & 0  & -       & (0,10) & 10 & 0.72    & (0,0)  & 10 & 4.78    & (0,0)  & 10 & 0.87    & (0,0)  & 10 & 0.97    & (0,0)  & 10 & 4.22    & (0,0)  & 10 & 0.5     & (0,0)  \\
			30  & 7 & 0  & -       & (0,10) & 0  & -       & (0,10) & 10 & 1.75    & (0,0)  & 10 & 6.24    & (0,0)  & 10 & 1.63    & (0,0)  & 10 & 2.8     & (0,0)  & 10 & 5.25    & (0,0)  & 10 & 0.97    & (0,0)  \\ \hline \Tstrut \Bstrut
			40  & 3 & 10 & 22.07   & (0,0)  & 10 & 30.81   & (0,0)  & 10 & 0.67    & (0,0)  & 10 & 1.69    & (0,0)  & 10 & 0.64    & (0,0)  & 10 & 0.66    & (0,0)  & 10 & 1.64    & (0,0)  & 10 & 0.52    & (0,0)  \\
			40  & 4 & 4  & 1685.96 & (0,6)  & 5  & 1988.99 & (0,5)  & 10 & 3.22    & (0,0)  & 10 & 20.57   & (0,0)  & 10 & 3.16    & (0,0)  & 10 & 3.79    & (0,0)  & 10 & 20.42   & (0,0)  & 10 & 1.88    & (0,0)  \\
			40  & 5 & 0  & -       & (1,9)  & 0  & -       & (0,10) & 10 & 11.66   & (0,0)  & 10 & 46.41   & (0,0)  & 10 & 8.05    & (0,0)  & 10 & 16.54   & (0,0)  & 10 & 40.95   & (0,0)  & 10 & 4.44    & (0,0)  \\
			40  & 6 & 0  & -       & (0,10) & 0  & -       & (0,10) & 10 & 37.16   & (0,0)  & 10 & 251.22  & (0,0)  & 10 & 39.25   & (0,0)  & 10 & 55.53   & (0,0)  & 10 & 240.71  & (0,0)  & 10 & 20.83   & (0,0)  \\
			40  & 7 & 0  & -       & (0,10) & 0  & -       & (0,10) & 10 & 89.58   & (0,0)  & 9  & 281.21  & (0,1)  & 10 & 86.99   & (0,0)  & 10 & 130.27  & (0,0)  & 9  & 250.89  & (0,1)  & 10 & 63.02   & (0,0)  \\ \hline \Tstrut \Bstrut
			50  & 3 & 10 & 36.97   & (0,0)  & 10 & 52.2    & (0,0)  & 10 & 2.25    & (0,0)  & 10 & 5.17    & (0,0)  & 10 & 1.71    & (0,0)  & 10 & 2.17    & (0,0)  & 10 & 5.33    & (0,0)  & 10 & 1.52    & (0,0)  \\
			50  & 4 & 0  & -       & (0,10) & 2  & 2841.42 & (0,8)  & 10 & 23.46   & (0,0)  & 10 & 134.57  & (0,0)  & 10 & 11.04   & (0,0)  & 10 & 28.59   & (0,0)  & 10 & 125.92  & (0,0)  & 10 & 6.74    & (0,0)  \\
			50  & 5 & 0  & -       & (0,10) & 0  & -       & (0,10) & 10 & 110.5   & (0,0)  & 10 & 1122.29 & (0,0)  & 10 & 64.76   & (0,0)  & 10 & 136.59  & (0,0)  & 10 & 996.69  & (0,0)  & 10 & 38.8    & (0,0)  \\
			50  & 6 & 0  & -       & (1,9)  & 0  & -       & (0,10) & 9  & 484.95  & (0,1)  & 5  & 1305.47 & (0,5)  & 10 & 507.23  & (0,0)  & 9  & 724.16  & (0,1)  & 5  & 1249.05 & (0,5)  & 10 & 400.63  & (0,0)  \\
			50  & 7 & 0  & -       & (0,10) & 0  & -       & (0,10) & 9  & 1737.94 & (0,1)  & 2  & 1654.82 & (0,8)  & 9  & 1195.03 & (0,1)  & 6  & 1514.52 & (0,4)  & 2  & 1716.35 & (0,8)  & 10 & 971.73  & (0,0)  \\ \hline \Tstrut \Bstrut
			60  & 3 & 10 & 107.96  & (0,0)  & 10 & 143.68  & (0,0)  & 10 & 9.04    & (0,0)  & 10 & 25.51   & (0,0)  & 10 & 4.6     & (0,0)  & 10 & 9.07    & (0,0)  & 10 & 23.45   & (0,0)  & 10 & 3.97    & (0,0)  \\
			60  & 4 & 0  & -       & (0,10) & 0  & -       & (0,10) & 10 & 169.13  & (0,0)  & 10 & 792.27  & (0,0)  & 10 & 51.03   & (0,0)  & 10 & 191.15  & (0,0)  & 10 & 784.7   & (0,0)  & 10 & 34.05   & (0,0)  \\
			60  & 5 & 0  & -       & (1,9)  & 0  & -       & (0,10) & 8  & 1439.71 & (0,2)  & 2  & 2158.86 & (0,8)  & 10 & 601.92  & (0,0)  & 6  & 1007.89 & (0,4)  & 1  & 834.29  & (0,9)  & 10 & 416.31  & (0,0)  \\
			60  & 6 & 0  & -       & (1,9)  & 0  & -       & (0,10) & 6  & 2722.71 & (0,4)  & 0  & -       & (0,10) & 6  & 1389.04 & (0,4)  & 2  & 1900.28 & (0,8)  & 0  & -       & (0,10) & 7  & 1124.37 & (0,3)  \\
			60  & 7 & 0  & -       & (0,10) & 0  & -       & (0,10) & 1  & 2762.75 & (0,9)  & 0  & -       & (0,10) & 3  & 2454.84 & (0,7)  & 0  & -       & (0,10) & 0  & -       & (0,10) & 5  & 2174.34 & (0,5)  \\ \hline \Tstrut \Bstrut
			70  & 3 & 10 & 271.77  & (0,0)  & 10 & 340.95  & (0,0)  & 10 & 28.85   & (0,0)  & 10 & 125.02  & (0,0)  & 10 & 11.85   & (0,0)  & 10 & 28.43   & (0,0)  & 10 & 126.12  & (0,0)  & 10 & 9.08    & (0,0)  \\
			70  & 4 & 0  & -       & (0,10) & 0  & -       & (0,10) & 10 & 1092.52 & (0,0)  & 4  & 1605.46 & (0,6)  & 10 & 257.09  & (0,0)  & 9  & 903.07  & (0,1)  & 4  & 1556.9  & (0,6)  & 10 & 163.71  & (0,0)  \\
			70  & 5 & 0  & -       & (1,9)  & 0  & -       & (0,10) & 2  & 3419.18 & (0,8)  & 0  & -       & (0,10) & 6  & 2052.85 & (0,4)  & 0  & -       & (0,10) & 0  & -       & (0,10) & 8  & 1554.53 & (0,2)  \\
			70  & 6 & 0  & -       & (0,10) & 0  & -       & (0,10) & 0  & -       & (0,10) & 0  & -       & (0,10) & 2  & 1998.62 & (0,8)  & 0  & -       & (0,10) & 0  & -       & (0,10) & 2  & 1165.76 & (0,8)  \\
			70  & 7 & 0  & -       & (1,9)  & 0  & -       & (0,10) & 0  & -       & (0,10) & 0  & -       & (0,10) & 1  & 3338.91 & (0,9)  & 0  & -       & (0,10) & 0  & -       & (0,10) & 2  & 2254.76 & (0,8)  \\ \hline \Tstrut \Bstrut
			80  & 3 & 10 & 401.49  & (0,0)  & 10 & 489.45  & (0,0)  & 10 & 95.62   & (0,0)  & 10 & 356     & (0,0)  & 10 & 29.46   & (0,0)  & 10 & 98.89   & (0,0)  & 10 & 351.56  & (0,0)  & 10 & 23.84   & (0,0)  \\
			80  & 4 & 0  & -       & (0,10) & 0  & -       & (0,10) & 4  & 846.36  & (0,6)  & 1  & 2285.8  & (0,9)  & 10 & 997.93  & (0,0)  & 4  & 912.5   & (0,6)  & 1  & 2396.79 & (0,9)  & 10 & 1012.99 & (0,0)  \\
			80  & 5 & 0  & -       & (3,7)  & 0  & -       & (0,10) & 0  & -       & (0,10) & 0  & -       & (0,10) & 4  & 2676.08 & (0,6)  & 0  & -       & (0,10) & 0  & -       & (0,10) & 4  & 2225.35 & (0,6)  \\
			80  & 6 & 0  & -       & (0,10) & 0  & -       & (0,10) & 0  & -       & (0,10) & 0  & -       & (0,10) & 0  & -       & (0,10) & 0  & -       & (0,10) & 0  & -       & (1,9)  & 0  & -       & (0,10) \\
			80  & 7 & 0  & -       & (0,10) & 0  & -       & (0,10) & 0  & -       & (0,10) & 0  & -       & (0,10) & 0  & -       & (0,10) & 0  & -       & (0,10) & 0  & -       & (1,9)  & 0  & -       & (0,10) \\ \hline \Tstrut \Bstrut
			90  & 3 & 10 & 778.15  & (0,0)  & 10 & 850.51  & (0,0)  & 10 & 283.73  & (0,0)  & 10 & 748.69  & (0,0)  & 10 & 71.74   & (0,0)  & 10 & 289.5   & (0,0)  & 10 & 750.2   & (0,0)  & 10 & 59.84   & (0,0)  \\
			90  & 4 & 0  & -       & (0,10) & 0  & -       & (0,10) & 1  & 2421.44 & (0,9)  & 0  & -       & (0,10) & 7  & 1597.38 & (0,3)  & 1  & 2720.16 & (0,9)  & 0  & -       & (0,10) & 8  & 1499.22 & (0,2)  \\
			90  & 5 & 0  & -       & (3,7)  & 0  & -       & (0,10) & 0  & -       & (0,10) & 0  & -       & (0,10) & 0  & -       & (0,10) & 0  & -       & (0,10) & 0  & -       & (0,10) & 0  & -       & (0,10) \\
			90  & 6 & 0  & -       & (0,10) & 0  & -       & (0,10) & 0  & -       & (0,10) & 0  & -       & (0,10) & 0  & -       & (0,10) & 0  & -       & (0,10) & 0  & -       & (1,9)  & 0  & -       & (0,10) \\
			90  & 7 & 0  & -       & (0,10) & 0  & -       & (0,10) & 0  & -       & (0,10) & 0  & -       & (0,10) & 0  & -       & (0,10) & 0  & -       & (0,10) & 0  & -       & (0,10) & 0  & -       & (0,10) \\ \hline \Tstrut \Bstrut
			100 & 3 & 8  & 1631.27 & (0,2)  & 8  & 1530.55 & (0,2)  & 10 & 967.47  & (0,0)  & 7  & 1684.11 & (0,3)  & 10 & 262.53  & (0,0)  & 10 & 1048.7  & (0,0)  & 7  & 1688.68 & (0,3)  & 10 & 212.64  & (0,0)  \\
			100 & 4 & 0  & -       & (0,10) & 0  & -       & (0,10) & 0  & -       & (0,10) & 0  & -       & (0,10) & 1  & 745.24  & (0,9)  & 0  & -       & (0,10) & 0  & -       & (0,10) & 2  & 1764.28 & (0,8)  \\
			100 & 5 & 0  & -       & (2,8)  & 0  & -       & (0,10) & 0  & -       & (0,10) & 0  & -       & (0,10) & 0  & -       & (0,10) & 0  & -       & (0,10) & 0  & -       & (0,10) & 0  & -       & (0,10) \\
			100 & 6 & 0  & -       & (2,8)  & 0  & -       & (0,10) & 0  & -       & (0,10) & 0  & -       & (0,10) & 0  & -       & (0,10) & 0  & -       & (0,10) & 0  & -       & (0,10) & 0  & -       & (0,10) \\
			100 & 7 & 0  & -       & (0,10) & 0  & -       & (0,10) & 0  & -       & (0,10) & 0  & -       & (0,10) & 0  & -       & (0,10) & 0  & -       & (0,10) & 0  & -       & (3,7)  & 0  & -       & (0,10)
		\end{tabular}
	\end{table}
	
\end{landscape}

\begin{table}[]
	\centering
	\tiny
	\caption{Set covering results}
	\label{tab:SetCoveringResults}
	\begin{tabular}{P{2ex}P{1.5ex}|P{1ex}P{6ex}P{6ex}|P{1ex}P{6ex}P{6ex}|P{1ex}P{6ex}P{6ex}|P{1ex}P{6ex}P{6ex}|P{1ex}P{6ex}P{6ex}|P{1ex}P{6ex}P{6ex}|P{1ex}P{6ex}P{6ex}|P{1ex}P{6ex}P{6ex}} 
		\multicolumn{2}{c}{}  & 
		\multicolumn{3}{c}{\texttt{K}} & 
		\multicolumn{3}{c}{\texttt{O}} & 
		\multicolumn{3}{c}{\texttt{TD}} & 
		\multicolumn{3}{c}{\texttt{BU}} & 
		\multicolumn{3}{c}{\texttt{Coup}}  
		\\ \hline \Tstrut \Bstrut
		$n$	&	$\nobj$ 	& $\mathsf{S}$	& $\bar{t}$	& $(\mathsf{M},\mathsf{T})$ 	& $\mathsf{S}$	& $\bar{t}$	& $(\mathsf{M},\mathsf{T})$  	& $\mathsf{S}$	& $\bar{t}$	& $(\mathsf{M},\mathsf{T})$  	& $\mathsf{S}$	& $\bar{t}$	& $(\mathsf{M},\mathsf{T})$  	& $\mathsf{S}$	& $\bar{t}$	& $(\mathsf{M},\mathsf{T})$  \\ \hline \hline \Bstrut\Tstrut
		100 & 3 & 10 & 8.33    & (0,0)  & 10 & 10.81   & (0,0)  & 10 & 10.14  & (0,0)  & 10 & 10.34   & (0,0)  & 10 & 10.18  & (0,0)    \\
		100 & 4 & 10 & 326.36  & (0,0)  & 10 & 427.21  & (0,0)  & 10 & 10.7   & (0,0)  & 10 & 11.32   & (0,0)  & 10 & 10.61  & (0,0)    \\
		100 & 5 & 1  & 105.35  & (0,9)  & 1  & 553.49  & (0,9)  & 10 & 11.27  & (0,0)  & 10 & 15.26   & (0,0)  & 10 & 10.82  & (0,0)    \\
		100 & 6 & 0  & -       & (0,10) & 0  & -       & (0,10) & 10 & 11.91  & (0,0)  & 10 & 35.35   & (0,0)  & 10 & 11.17  & (0,0)    \\
		100 & 7 & 0  & -       & (0,10) & 0  & -       & (0,10) & 10 & 24.09  & (0,0)  & 10 & 54.9    & (0,0)  & 10 & 16.01  & (0,0)    \\ \hline \Tstrut \Bstrut
		150 & 3 & 10 & 34.6    & (0,0)  & 10 & 44.22   & (0,0)  & 10 & 36.82  & (0,0)  & 10 & 59.11   & (0,0)  & 10 & 36.27  & (0,0)    \\
		150 & 4 & 7  & 1290.86 & (0,3)  & 9  & 1485.95 & (0,1)  & 10 & 83.76  & (0,0)  & 8  & 197.08  & (0,2)  & 9  & 62.88  & (0,1)    \\
		150 & 5 & 0  & -       & (0,10) & 0  & -       & (0,10) & 7  & 196.1  & (0,3)  & 5  & 652.01  & (0,5)  & 7  & 136.13 & (0,3)    \\
		150 & 6 & 0  & -       & (0,10) & 0  & -       & (0,10) & 8  & 539.84 & (0,2)  & 5  & 803.55  & (0,5)  & 9  & 315.13 & (0,1)    \\
		150 & 7 & 0  & -       & (0,10) & 0  & -       & (0,10) & 4  & 783.5  & (0,6)  & 3  & 1392.27 & (0,7)  & 4  & 454.3  & (1,5)    \\ \hline \Tstrut \Bstrut
		200 & 3 & 10 & 97.47   & (0,0)  & 10 & 124.17  & (0,0)  & 1  & 63.42  & (0,9)  & 1  & 202.83  & (0,9)  & 1  & 158.17 & (0,9)    \\
		200 & 4 & 0  & -       & (0,10) & 0  & -       & (0,10) & 1  & 292.97 & (0,9)  & 0  & -       & (0,10) & 0  & -      & (0,10)  \\
		200 & 5 & 0  & -       & (2,8)  & 0  & -       & (0,10) & 0  & -      & (0,10) & 0  & -       & (0,10) & 0  & -      & (0,10)  \\
		200 & 6 & 0  & -       & (0,10) & 0  & -       & (0,10) & 0  & -      & (0,10) & 0  & -       & (0,10) & 0  & -      & (0,10)  \\
		200 & 7 & 0  & -       & (0,10) & 0  & -       & (0,10) & 0  & -      & (0,10) & 0  & -       & (0,10) & 0  & -      & (0,10)  
	\end{tabular}

\bigskip

	\caption{Set packing results}
\begin{tabular}{P{2ex}P{1.5ex}|P{1ex}P{6ex}P{6ex}|P{1ex}P{6ex}P{6ex}|P{1ex}P{6ex}P{6ex}|P{1ex}P{6ex}P{6ex}|P{1ex}P{6ex}P{6ex}|P{1ex}P{6ex}P{6ex}|P{1ex}P{6ex}P{6ex}|P{1ex}P{6ex}P{6ex}} 
	\multicolumn{2}{c}{}  & 
	\multicolumn{3}{c}{\texttt{K}} & 
	\multicolumn{3}{c}{\texttt{O}} & 
	\multicolumn{3}{c}{\texttt{TD}} & 
	\multicolumn{3}{c}{\texttt{BU}} & 
	\multicolumn{3}{c}{\texttt{Coup}} 
	\\ \hline \Tstrut \Bstrut
	$n$	&	$\nobj$ 	& $\mathsf{S}$	& $\bar{t}$	& $(\mathsf{M},\mathsf{T})$ 	& $\mathsf{S}$	& $\bar{t}$	& $(\mathsf{M},\mathsf{T})$  	& $\mathsf{S}$	& $\bar{t}$	& $(\mathsf{M},\mathsf{T})$  	& $\mathsf{S}$	& $\bar{t}$	& $(\mathsf{M},\mathsf{T})$  	& $\mathsf{S}$	& $\bar{t}$	& $(\mathsf{M},\mathsf{T})$  \\ \hline \hline \Bstrut\Tstrut
	100 & 3 & 10 & 9.38    & (0,0)  & 10 & 14.34   & (0,0)  & 10 & 0.09   & (0,0)  & 10 & 0.12   & (0,0)  & 10 & 0.08   & (0,0)    \\
	100 & 4 & 10 & 518.46  & (0,0)  & 10 & 954.85  & (0,0)  & 10 & 0.15   & (0,0)  & 10 & 0.23   & (0,0)  & 10 & 0.13   & (0,0)    \\
	100 & 5 & 0  & -       & (0,10) & 0  & -       & (0,10) & 10 & 0.28   & (0,0)  & 10 & 0.56   & (0,0)  & 10 & 0.24   & (0,0)    \\
	100 & 6 & 0  & -       & (0,10) & 0  & -       & (0,10) & 10 & 0.77   & (0,0)  & 10 & 0.86   & (0,0)  & 10 & 0.57   & (0,0)    \\
	100 & 7 & 0  & -       & (0,10) & 0  & -       & (0,10) & 10 & 3.6    & (0,0)  & 10 & 3.73   & (0,0)  & 10 & 1.28   & (0,0)    \\
	150 & 3 & 10 & 23.21   & (0,0)  & 10 & 36.33   & (0,0)  & 10 & 12.11  & (0,0)  & 10 & 17.94  & (0,0)  & 10 & 12.41  & (0,0)    \\
	150 & 4 & 1  & 2916.26 & (0,9)  & 2  & 3057.22 & (0,8)  & 10 & 27.97  & (0,0)  & 10 & 75.58  & (0,0)  & 10 & 23.16  & (0,0)    \\
	150 & 5 & 0  & -       & (2,8)  & 0  & -       & (0,10) & 10 & 87.74  & (0,0)  & 5  & 155.74 & (0,5)  & 10 & 65.68  & (0,0)    \\
	150 & 6 & 0  & -       & (1,9)  & 0  & -       & (0,10) & 10 & 203.97 & (0,0)  & 5  & 318.31 & (0,5)  & 10 & 125.62 & (0,0)    \\
	150 & 7 & 0  & -       & (0,10) & 0  & -       & (0,10) & 10 & 888.75 & (0,0)  & 2  & 196.25 & (0,8)  & 10 & 475.53 & (0,0)    \\
	200 & 3 & 10 & 76.47   & (0,0)  & 10 & 111.18  & (0,0)  & 1  & 121.92 & (0,9)  & 0  & -      & (0,10) & 1  & 137.78 & (0,9)    \\
	200 & 4 & 0  & -       & (0,10) & 0  & -       & (0,10) & 0  & -      & (0,10) & 0  & -      & (0,10) & 1  & 336.82 & (0,9)   \\
	200 & 5 & 0  & -       & (0,10) & 0  & -       & (0,10) & 0  & -      & (0,10) & 0  & -      & (0,10) & 0  & -      & (0,10)  \\
	200 & 6 & 0  & -       & (1,9)  & 0  & -       & (0,10) & 0  & -      & (0,10) & 0  & -      & (0,10) & 0  & -      & (0,10)  \\
	200 & 7 & 0  & -       & (0,10) & 0  & -       & (0,10) & 0  & -      & (0,10) & 0  & -      & (0,10) & 0  & -      & (0,10) 
\end{tabular}
\label{tab:SetPackingResults}

\end{table}

\begin{landscape}	
	\begin{table}[]
		\centering
		\tiny
		\caption{Absolute value results}
		\label{tab:AbsValueResults}
		\begin{tabular}{P{1ex}P{1ex}P{2ex}|P{1ex}P{6ex}P{6ex}|P{1ex}P{6ex}P{6ex}|P{1ex}P{6ex}P{6ex}|P{1ex}P{6ex}P{6ex}|P{1ex}P{6ex}P{6ex}} 
			\multicolumn{3}{c}{}  & 
			\multicolumn{3}{c}{\texttt{K}} & 
			\multicolumn{3}{c}{\texttt{O}} & 
			\multicolumn{3}{c}{\texttt{TD}} & 
			\multicolumn{3}{c}{\texttt{BU}} & 
			\multicolumn{3}{c}{\texttt{Coup}} 
			\\ \hline \Tstrut \Bstrut
			$n$	&	$\nobj$  & $\mathsf{C}$	& $\mathsf{S}$	& $\bar{t}$	& $(\mathsf{M},\mathsf{T})$ 	& $\mathsf{S}$	& $\bar{t}$	& $(\mathsf{M},\mathsf{T})$  	& $\mathsf{S}$	& $\bar{t}$	& $(\mathsf{M},\mathsf{T})$  	& $\mathsf{S}$	& $\bar{t}$	& $(\mathsf{M},\mathsf{T})$  	& $\mathsf{S}$	& $\bar{t}$	& $(\mathsf{M},\mathsf{T})$    \\ \hline \hline \Bstrut\Tstrut
			15 & 3 & 30 & 10 & 1.21   & (0,0)  & 10 & 1.71   & (0,0)  & 10 & 0      & (0,0) & 10 & 0      & (0,0) & 10 & 0      & (0,0) \\
			15 & 4 & 30 & 10 & 8.31   & (0,0)  & 10 & 25.9   & (0,0)  & 10 & 0      & (0,0) & 10 & 0      & (0,0) & 10 & 0      & (0,0) \\
			15 & 5 & 30 & 4  & 1882.3 & (0,6)  & 9  & 2166.9 & (0,1)  & 10 & 0      & (0,0) & 10 & 0      & (0,0) & 10 & 0      & (0,0) \\
			15 & 6 & 30 & 0  & -      & (1,9)  & 0  & -      & (0,10) & 10 & 0      & (0,0) & 10 & 0      & (0,0) & 10 & 0      & (0,0) \\
			15 & 7 & 30 & 0  & -      & (0,10) & 0  & -      & (0,10) & 10 & 0      & (0,0) & 10 & 0      & (0,0) & 10 & 0      & (0,0) \\ \hline \Bstrut\Tstrut
			15 & 3 & 40 & 10 & 2.05   & (0,0)  & 10 & 2.85   & (0,0)  & 10 & 0.01   & (0,0) & 10 & 0.01   & (0,0) & 10 & 0.01   & (0,0) \\
			15 & 4 & 40 & 10 & 23.02  & (0,0)  & 10 & 65.87  & (0,0)  & 10 & 0.01   & (0,0) & 10 & 0.01   & (0,0) & 10 & 0.01   & (0,0) \\
			15 & 5 & 40 & 0  & -      & (0,10) & 0  & -      & (0,10) & 10 & 0.01   & (0,0) & 10 & 0.01   & (0,0) & 10 & 0.01   & (0,0) \\
			15 & 6 & 40 & 0  & -      & (1,9)  & 0  & -      & (0,10) & 10 & 0.02   & (0,0) & 10 & 0.02   & (0,0) & 10 & 0.02   & (0,0) \\
			15 & 7 & 40 & 0  & -      & (0,10) & 0  & -      & (0,10) & 10 & 0.04   & (0,0) & 10 & 0.04   & (0,0) & 10 & 0.04   & (0,0) \\ \hline \Bstrut\Tstrut
			15 & 3 & 50 & 10 & 2.76   & (0,0)  & 10 & 3.71   & (0,0)  & 10 & 0.01   & (0,0) & 10 & 0.01   & (0,0) & 10 & 0.02   & (0,0) \\
			15 & 4 & 50 & 10 & 31.58  & (0,0)  & 10 & 87.62  & (0,0)  & 10 & 0.01   & (0,0) & 10 & 0.02   & (0,0) & 10 & 0.02   & (0,0) \\
			15 & 5 & 50 & 0  & -      & (1,9)  & 0  & -      & (0,10) & 10 & 0.02   & (0,0) & 10 & 0.02   & (0,0) & 10 & 0.02   & (0,0) \\
			15 & 6 & 50 & 0  & -      & (1,9)  & 0  & -      & (0,10) & 10 & 0.03   & (0,0) & 10 & 0.04   & (0,0) & 10 & 0.03   & (0,0) \\
			15 & 7 & 50 & 0  & -      & (0,10) & 0  & -      & (0,10) & 10 & 0.07   & (0,0) & 10 & 0.09   & (0,0) & 10 & 0.05   & (0,0) \\ \hline \Bstrut\Tstrut
			20 & 3 & 30 & 10 & 7.86   & (0,0)  & 10 & 9.84   & (0,0)  & 10 & 0.14   & (0,0) & 10 & 0.14   & (0,0) & 10 & 0.17   & (0,0) \\
			20 & 4 & 30 & 10 & 89.82  & (0,0)  & 10 & 222.09 & (0,0)  & 10 & 0.15   & (0,0) & 10 & 0.14   & (0,0) & 10 & 0.18   & (0,0) \\
			20 & 5 & 30 & 0  & -      & (1,9)  & 0  & -      & (0,10) & 10 & 0.17   & (0,0) & 10 & 0.19   & (0,0) & 10 & 0.21   & (0,0) \\
			20 & 6 & 30 & 0  & -      & (2,8)  & 0  & -      & (0,10) & 10 & 0.24   & (0,0) & 10 & 0.28   & (0,0) & 10 & 0.29   & (0,0) \\
			20 & 7 & 30 & 0  & -      & (1,9)  & 0  & -      & (0,10) & 10 & 0.4    & (0,0) & 10 & 0.49   & (0,0) & 10 & 0.38   & (0,0) \\ \hline \Bstrut\Tstrut
			20 & 3 & 40 & 10 & 32.85  & (0,0)  & 10 & 40.07  & (0,0)  & 10 & 0.63   & (0,0) & 10 & 0.74   & (0,0) & 10 & 0.94   & (0,0) \\
			20 & 4 & 40 & 10 & 553.77 & (0,0)  & 10 & 1046.7 & (0,0)  & 10 & 0.7    & (0,0) & 10 & 0.76   & (0,0) & 10 & 1.06   & (0,0) \\
			20 & 5 & 40 & 0  & -      & (1,9)  & 0  & -      & (0,10) & 10 & 0.8    & (0,0) & 10 & 0.9    & (0,0) & 10 & 1.09   & (0,0) \\
			20 & 6 & 40 & 0  & -      & (0,10) & 0  & -      & (0,10) & 10 & 1.28   & (0,0) & 10 & 1.43   & (0,0) & 10 & 1.56   & (0,0) \\
			20 & 7 & 40 & 0  & -      & (1,9)  & 0  & -      & (0,10) & 10 & 2.62   & (0,0) & 10 & 2.67   & (0,0) & 10 & 2.41   & (0,0) \\ \hline \Bstrut\Tstrut
			20 & 3 & 50 & 10 & 42.55  & (0,0)  & 10 & 55.92  & (0,0)  & 10 & 1.37   & (0,0) & 10 & 1.56   & (0,0) & 10 & 1.98   & (0,0) \\
			20 & 4 & 50 & 10 & 933.77 & (0,0)  & 10 & 1835.4 & (0,0)  & 10 & 1.47   & (0,0) & 10 & 1.69   & (0,0) & 10 & 2.47   & (0,0) \\
			20 & 5 & 50 & 0  & -      & (0,10) & 0  & -      & (0,10) & 10 & 1.84   & (0,0) & 10 & 2.15   & (0,0) & 10 & 2.78   & (0,0) \\
			20 & 6 & 50 & 0  & -      & (2,8)  & 0  & -      & (0,10) & 10 & 3.29   & (0,0) & 10 & 3.5    & (0,0) & 10 & 3.95   & (0,0) \\
			20 & 7 & 50 & 0  & -      & (0,10) & 0  & -      & (0,10) & 10 & 7.08   & (0,0) & 10 & 7.03   & (0,0) & 10 & 6.26   & (0,0) \\ \hline \Bstrut\Tstrut
			25 & 3 & 30 & 10 & 57.61  & (0,0)  & 10 & 69.52  & (0,0)  & 10 & 3.12   & (0,0) & 10 & 3.26   & (0,0) & 10 & 4.21   & (0,0) \\
			25 & 4 & 30 & 10 & 1209.4 & (0,0)  & 9  & 2037.7 & (0,1)  & 10 & 3.35   & (0,0) & 10 & 3.59   & (0,0) & 10 & 4.43   & (0,0) \\
			25 & 5 & 30 & 0  & -      & (1,9)  & 0  & -      & (0,10) & 10 & 3.78   & (0,0) & 10 & 3.98   & (0,0) & 10 & 5.05   & (0,0) \\
			25 & 6 & 30 & 0  & -      & (0,10) & 0  & -      & (0,10) & 10 & 5.69   & (0,0) & 10 & 5.37   & (0,0) & 10 & 7.08   & (0,0) \\
			25 & 7 & 30 & 0  & -      & (0,10) & 0  & -      & (0,10) & 10 & 10.54  & (0,0) & 10 & 8.97   & (0,0) & 10 & 11.03  & (0,0) \\ \hline \Bstrut\Tstrut
			25 & 3 & 40 & 10 & 86.15  & (0,0)  & 10 & 89.88  & (0,0)  & 10 & 24.76  & (0,0) & 10 & 25.68  & (0,0) & 10 & 41.76  & (0,0) \\
			25 & 4 & 40 & 8  & 2974   & (0,2)  & 2  & 3461.5 & (0,8)  & 10 & 26.58  & (0,0) & 10 & 27.51  & (0,0) & 10 & 43.84  & (0,0) \\
			25 & 5 & 40 & 0  & -      & (0,10) & 0  & -      & (0,10) & 10 & 32.49  & (0,0) & 10 & 33.66  & (0,0) & 10 & 51.75  & (0,0) \\
			25 & 6 & 40 & 0  & -      & (0,10) & 0  & -      & (0,10) & 10 & 60.98  & (0,0) & 10 & 51.1   & (0,0) & 10 & 77.37  & (0,0) \\
			25 & 7 & 40 & 0  & -      & (0,10) & 0  & -      & (0,10) & 10 & 148.45 & (0,0) & 10 & 95.4   & (0,0) & 10 & 144.34 & (0,0) \\ \hline \Bstrut\Tstrut
			25 & 3 & 50 & 10 & 87.75  & (0,0)  & 10 & 86.78  & (0,0)  & 10 & 49.21  & (0,0) & 10 & 53.73  & (0,0) & 10 & 99.72  & (0,0) \\
			25 & 4 & 50 & 6  & 2144.1 & (0,4)  & 2  & 2135.3 & (0,8)  & 10 & 54.91  & (0,0) & 10 & 60.58  & (0,0) & 10 & 106.52 & (0,0) \\
			25 & 5 & 50 & 0  & -      & (0,10) & 0  & -      & (0,10) & 10 & 67.79  & (0,0) & 10 & 72.88  & (0,0) & 10 & 128.78 & (0,0) \\
			25 & 6 & 50 & 0  & -      & (0,10) & 0  & -      & (0,10) & 10 & 137.12 & (0,0) & 10 & 111.41 & (0,0) & 10 & 171.18 & (0,0) \\
			25 & 7 & 50 & 0  & -      & (0,10) & 0  & -      & (0,10) & 10 & 427.1  & (0,0) & 10 & 222.51 & (0,0) & 10 & 358.04 & (0,0)
		\end{tabular}
	\label{tab:AbsoluteValueResults}
	\end{table}
	
\end{landscape}

\end{appendix}

\end{document}